\documentclass[11pt,reqno]{amsart} 
\usepackage[a4paper, margin=3cm]{geometry}

\usepackage{graphicx}
\graphicspath{ {./images/} }
\usepackage{lipsum}
\usepackage{subcaption}

\usepackage{xcolor}

\usepackage{amsthm}

\usepackage{bbm}
\usepackage{mathrsfs}
\usepackage{dynkin-diagrams}

\usepackage[utf8]{inputenc}
\usepackage{xypic}

\usepackage{tikz-cd}

\usepackage{enumitem}

\usepackage{amssymb}
\usepackage{float}
\usepackage{amsbsy}
\usepackage[all]{xy}\usepackage{xypic}
\usepackage{braket}
\usepackage[colorlinks = true,citecolor=black]{hyperref}
\usepackage{amsmath}

\usepackage{hyperref}

\newcommand\myeq{\mathrel{\stackrel{\makebox[0pt]{\mbox{\normalfont\tiny loc}}}{=}}}

\hypersetup{
     colorlinks=true,
     linkcolor=blue,
     filecolor=blue,
     citecolor = blue,      
     urlcolor=cyan,
     }

\makeatletter
\newcommand\opteq[1]{\mathrel{\mathpalette\opt@eq{#1}}}
\newcommand{\opt@eq}[2]{%
  \begingroup
  \sbox\z@{$#1#2$}%
  \sbox\tw@{\resizebox{!}{.5\ht\z@}{$\m@th#1($}}%
  \nonscript\hskip-\wd\tw@
  \mkern1mu
  \raisebox{-.35\ht\z@}[0pt][0pt]{\resizebox{!}{.5\ht\z@}{$\m@th#1($}}%
  \mkern-1mu
  {#2}%
  \mkern-1mu
  \raisebox{-.35\ht\z@}[0pt][0pt]{\resizebox{!}{.5\ht\z@}{$\m@th#1)$}}%
  \mkern1mu
  \nonscript\hskip-\wd\tw@
  \endgroup
}
\makeatother

\newtheorem{theorem}{Theorem}[section]
\newtheorem*{theorem*}{Theorem}
\newtheorem{theorem-non}{Theorem}
\newtheorem{lemma-non}{Lemma}

\theoremstyle{definition} 
\newtheorem{thm}{Theorem}

\theoremstyle{definition} 
\newtheorem{corollarynon}{Corollary}

\newtheorem{pbl}{Problem}

\newtheorem{conjecture-non}{Conjecture}

\newtheorem{corollary-non}{Corollary}
\newtheorem{proposition}[theorem]{Proposition}
\newtheorem{lemma}[theorem]{Lemma}
\newtheorem*{lemma*}{Lemma}
\newtheorem{corollary}[theorem]{Corollary}

\newtheorem*{conjecture*}{Conjecture}

\theoremstyle{definition}
\newtheorem{definition}[theorem]{Definition}

\theoremstyle{remark}
\newtheorem{remark}[theorem]{Remark}

\DeclareMathOperator{\im}{im}

\DeclareMathOperator{\rank}{rank}

\numberwithin{equation}{section}

\begin{document}
\title[Eder M. Correa]{$t$-Gauduchon Ricci-flat metrics on non-K\"{a}hler Calabi-Yau manifolds}

\author{Eder M. Correa}


\address{{IMECC-Unicamp, Departamento de Matem\'{a}tica. Rua S\'{e}rgio Buarque de Holanda 651, Cidade Universit\'{a}ria Zeferino Vaz. 13083-859, Campinas-SP, Brazil}}
\address{E-mail: {\rm ederc@unicamp.br}}

\begin{abstract} 
We construct new examples of $t$-Gauduchon Ricci-flat metrics, for all $t<1$, on compact non-K\"{a}hler Calabi-Yau manifolds defined by certain principal torus bundles over rational homogeneous varieties with Picard number $\varrho(X) > 1$. As an application, we provide a detailed description of new examples of Strominger-Bismut Ricci-flat Hermitian metrics, Lichnerowicz Ricci-flat Hermitian metrics, and balanced Hermitian metrics on principal $T^{2}$-bundles over the Fano threefold ${\mathbbm{P}}(T_{{\mathbbm{P}^{2}}})$.
\end{abstract}

\maketitle

\hypersetup{linkcolor=black}
\tableofcontents

\hypersetup{linkcolor=black}
\section{Introduction}

It is known that the infinite-dimensional affine space of Hermitian connections on the tangent
bundle of a Hermitian manifold $(M, g, J)$ supports a distinguished one-parameter family of Hermitian connections (a.k.a. canonical connections \cite{gauduchon11hermitian}) canonically depending on the complex
structure $J$ and the Riemannian metric $g$. Denoting by $\nabla^{t}$ this one-parameter family of Hermitian connections, we highlight the cases: 
\begin{enumerate}
\item The {\textit{Chern connection}} $\nabla^{Ch} = \nabla^{1}$: it is the unique connection compatible with the Hermitian metric and the holomorphic structure \cite{chern1967complex};
\item The {\textit{Lichnerowicz connection}} $\widehat{\nabla}^{LC} = \nabla^{0}$: it is defined by the restriction of the complexified Levi-Civita connection to $T^{1,0}M$ \cite{lichnerowicz1955theorie};
\item The {\textit{Strominger-Bismut connection}} $\nabla^{SB} = \nabla^{-1}$: it is the unique Hermitian connection with totally
skew-symmetric torsion \cite{Bismut1989}, \cite{strominger1986superstrings}.
\end{enumerate}
As in the case of Riemannian manifolds, we would like to understand the differential geometry of each of these connections in terms of their curvature properties. The relations between the curvature tensors of $\nabla^{Ch}$, $\nabla^{SB}$, and $\widehat{\nabla}^{LC}$ have been extensively studied, see for instance \cite{Gray1}, \cite{Tricerri}, \cite{Gaud}, \cite{Apostolov}, \cite{Liu1, Liu2}, \cite{Yang2, Yang1}, \cite{Angella}, \cite{WangYang}, \cite{WangYangZheng}, \cite{HeLiuYang}, \cite{Broder}. An important problem in this context is related to curvature conditions that generalize the concept of Calabi-Yau manifolds\footnote{In this paper, a Calabi-Yau manifold is a complex manifold with $c_{1}(M) = 0$.}to the non-K\"{a}hler context. More precisely, the problem can be formulated as follows. Let $(M,g,J)$ be a compact Hermitian manifold. Since each Hermitian connection $\nabla^{t}$ on the holomorphic tangent bundle induces a Hermitian connection on the anti-canonical line bundle ${\bf{K}}_{M}^{-1}$, denoting by $R(\nabla^{t})$ the curvature tensor of $\nabla^{t}$ and denoting by $\rho(\Omega,t)=\sqrt{-1}{\rm{tr}}(R(\nabla^{t}))$ its associated Ricci form, such that $\Omega = g(J \otimes {\rm{id}})$, from \cite[Remark 5]{gauduchon11hermitian}, see also \cite{GRANTCHAROV}, we have 
\begin{equation}
\label{differenceRicci}
\rho(\Omega,t) - \rho(\Omega,u) = \frac{t-u}{2}{\rm{d}}\delta \Omega,
\end{equation}
for every $t,u \in \mathbbm{R}$, such that $\delta = -\ast  {\rm{d}}  \ast$ is the codifferential associated with $g$. In this setting, denoting by $\rho_{1}(\Omega,t)$ the $(1,1)$-component of $\rho(\Omega,t)$, we can show that 
\begin{equation}
\rho_{1}(\Omega,t) =  \rho_{1}(\Omega,1)+ \frac{t-1}{2}\big ( \partial \partial^{\ast} \Omega + \bar{\partial}\bar{\partial}^{\ast}\Omega\big),
\end{equation}
where $\rho_{1}(\Omega,1) \myeq -\sqrt{-1} \partial \bar{\partial} \log(\det(\Omega))$, see for instance \cite[Proposition 4.1]{fu2022scalar}. From above, we say that $(M,g,J)$ is $t$-{\textit{Gauduchon Ricci-flat}} \cite{WangYang}, for some $t \in \mathbbm{R}$, if
\begin{equation}
\label{sRicciflat}
\rho_{1}(\Omega,t) = 0.
\end{equation}
We notice that Eq. (\ref{sRicciflat}) implies $c_{1}^{AC}(M) = 0$ (first {\textit{Aeppli-Chern class}, e.g. \cite{angella2014cohomological}), since 
\begin{equation}
c_{1}^{AC}(M) = \bigg [ \frac{\rho_{1}(\Omega,t)}{2\pi}\bigg] \in H^{1,1}_{A}(M) := \frac{\ker(\partial \bar{\partial}) \cap \Omega^{1,1}(M)}{\im(\partial) \cap \Omega^{1,1}(M) + \im(\bar{\partial}) \cap \Omega^{1,1}(M)}.
\end{equation}
From above, inspired by the Calabi-Yau theorem \cite{calabi2015kahler}, \cite{yau1977calabi,yau1978ricci}, and by \cite[Problem 1.8]{Liu2}, we have the following problem:
\begin{pbl}
\label{problem1}
On a compact complex manifold $(M,J)$, s.t. $c_{1}^{AC}(M) = 0$ (or\footnote{It is worth pointing out that $c_{1}(M) = 0 \Rightarrow c_{1}^{AC}(M) = 0$, e.g. \cite[Corllary 1.3]{Liu2}.} $c_{1}(M) = 0$), does there exist a $t$-Gauduchon Ricci-flat Hermitian metric $\Omega$, for some $t \in \mathbbm{R}$?
\end{pbl}
In general, for example, when $\Omega$ is non-K\"{a}hler, finding solutions to Eq. (\ref{sRicciflat}) is a highly non-trivial and challenging problem since the underlying partial differential equation is non-Monge-Amp\`{e}re type and contains non-elliptic terms. We observe that Problem 1 has a positive answer if 
\begin{enumerate}
\item $t=1$ and $M$ is a compact K\"{a}hler manifold \cite{yau1977calabi,yau1978ricci};
\item $t = 0$ and $M$ is a compact Kähler Calabi-Yau surface or a Hopf surface \cite{HeLiuYang};
\item $t < 1$ and $M$ is a suspension of a compact Sasaki-Einstein manifold \cite{Broder}, see also \cite{Liu2}, \cite{WangYang}, and \cite{correa2023levi};
\item $t=-1$ and $M$ is a total space of a principal $T^{2r}$-bundle over a K\"{a}hler-Einstein Fano manifold \cite{GRANTCHAROV, grantcharov2011geometry}.
\end{enumerate}
Motivated by Problem \ref{problem1}, from the construction provided in \cite{GRANTCHAROV}, and the ideas introduced in \cite{poddar2018group}, \cite{correa2023deformed}, we prove the following theorem.
\begin{thm}
\label{theoremA}
Let $X$ be a complex flag variety with Picard number $\varrho(X)>1$. Then, for every $k \in \mathbbm{Z}^{\times}$ and every real number $t < 1$, there exists $\lambda(k,t) \in \mathbbm{R}_{>0}$, and a short exact sequence of holomorphic vector bundles 
\begin{center}
\begin{tikzcd} 0 \arrow[r] & \mathscr{O}_{X}(k)  \arrow[r] & {\bf{E}}  \arrow[r]  & {\bf{F}} \arrow[r] & 0 \end{tikzcd}
\end{center}
satisfying the following properties:
\begin{enumerate}
\item[(1)] The holomorphic vector bundle ${\bf{F}} \to X$ admits a Hermitian structure ${\bf{h}}$ solving the Hermitian Yang-Mills equation
\begin{equation}
\label{HYM}
\sqrt{-1}\Lambda_{\omega_{0}}(F({\bf{h}})) = 0,
\end{equation}
where $[\omega_{0}] = \lambda(k,t)c_{1}(X)$. In particular, ${\bf{F}}$ is $[\omega_{0}]$-polystable;
\item[(2)] The manifold underlying the unitary frame bundle ${\rm{U}}({\bf{E}})$ is a compact non-K\"{a}hler Calabi-Yau manifold which admits a $t$-Gauduchon Ricci-flat Hermitian metric $\Omega$, such that the natural projection map
\begin{equation}
({\rm{U}}({\bf{E}}),\Omega) \to (X,\omega_{0}),
\end{equation}
is a Hermitian submersion;
\item[(3)] If $k = - {\bf{I}}(X)$, where ${\bf{I}}(X)$ is the Fano index of $X$, then the complex manifold underlying $\mathscr{O}_{X}(k)$ admits a complete Calabi-Yau metric.
\end{enumerate}
\end{thm}
The above theorem shows that, in the particular setting of rational homogeneous varieties, the result given in \cite[Proposition 9]{GRANTCHAROV} can be extended to all canonical connections $\nabla^{t}$, such that $t \in (-\infty,1)$. In particular, it highlights some aspects of Yang-Mills theory underlying the ideas introduced in \cite{GRANTCHAROV} and \cite{poddar2018group}. Also, it gives a positive answer for the question posed in Problem \ref{problem1} for certain classes of compact Hermitian manifolds defined by principal torus bundles. Further, we obtain from Theorem \ref{theoremA} a huge class of Hermitian connections $\nabla^{t}$ satisfying 
\begin{equation}
{\rm{Hol}}^{0}(\nabla^{t}) \subseteq {\rm{SU}}(n).
\end{equation}
It is worth to point out that, the result given in item (3) of Theorem \ref{theoremA} is a consequence of the well-known Calabi ansatz technique \cite{calabi1979metriques}.

The key point to solve the $t$-Gauduchon Ricci-flat equation in the context of Theorem \ref{theoremA} is the description provided in \cite{correa2023deformed} for the $\mathbbm{Z}$-module of primitive $(1,1)$-forms
\begin{equation}
H^{2}_{\omega_{0}}(X,\mathbbm{Z})_{{\text{prim}}} := \ker \Big (\Lambda_{\omega_{0}} \colon H^{2}(X,\mathbbm{R}) \to \mathbbm{R} \Big ) \cap H^{2}(X,\mathbbm{Z}),
\end{equation}
when $X$ is a rational homogeneous variety equipped with some $G$-invariant K\"{a}hler metric $\omega_{0}$. As a consequence of Theorem \ref{theoremA}, we obtain the following result.
\begin{corollarynon}
In the setting of the previous theorem, we have the following:
\begin{enumerate}
\item If $t = -1$, then $({\rm{U}}({\bf{E}}),\Omega)$ is Strominger-Bismut Ricci-flat; 
\item If $t = 0$, then $({\rm{U}}({\bf{E}}),\Omega)$ is Lichnerowicz Ricci-flat.
\end{enumerate}
\end{corollarynon}
In the setting of item (1), it follows that ${\rm{Hol}}^{0}(\nabla^{SB}) \subseteq {\rm{SU}}(n)$. The class of Hermitian manifolds given in item (1) are known in the literature as {\textit{Calabi-Yau with torsion}}, these manifolds appear in the study of heterotic string theory \cite{hull1986superstring}, \cite{strominger1986superstrings}, \cite{li2005existence}. The description provided in \cite{correa2023deformed} for the $\mathbbm{Z}$-module of primitive $(1,1)$-forms also allows us to obtain several examples of compact Hermitian manifolds $(M,\Omega,J)$ satisfying the balanced condition, i.e.,
\begin{equation}
{\rm{d}}\Omega^{n-1} = 0.
\end{equation}
Balanced metrics were introduced in \cite{michelsohn1982existence} and provide a generalization of K\"{a}hler metrics, it has been extensively studied, see for instance \cite{fino2023balanced} \cite{bedulli2017parabolic}, \cite{fino2019astheno}, \cite{fu2012balanced}, \cite{phong2019flow}, \cite{alonso2022existence}, and references therein. In this setting, we prove the following.
\begin{thm}
Let $X$ be a complex flag variety with Picard number $\varrho(X)>1$. Then, for every $r \in \mathbbm{Z}_{>0}$ and every $G$-invariant K\"{a}hler metric $\omega_{0}$, there exists a holomorphic vector bundle ${\bf{F}} \to X$, such that $\rank({\bf{F}}) = 2r$, satisfying the following properties:
\begin{enumerate}
\item ${\bf{F}} \to X$ admits a Hermitian structure ${\bf{h}}$ solving the Hermitian Yang-Mills equation
\begin{equation}
\sqrt{-1}\Lambda_{\omega_{0}}(F({\bf{h}})) = 0.
\end{equation}
 In particular, ${\bf{F}}$ is $[\omega_{0}]$-polystable;
\item The manifold underlying the unitary frame bundle ${\rm{U}}({\bf{F}})$ is a compact complex non-K\"{a}hler manifold which admits a balanced Hermitian metric $\Omega$, such that the natural projection map
\begin{equation}
({\rm{U}}({\bf{F}}),\Omega) \to (X,\omega_{0}),
\end{equation}
is a Hermitian submersion.
\end{enumerate}
\end{thm}

In order to illustrate the results presented above, by means of basic tools of Lie theory, we provide a detailed description of new examples of Strominger-Bismut Ricci-flat Hermitian metrics, Lichnerowicz Ricci-flat Hermitian metrics, and balanced Hermitian metrics on principal $T^{2}$-bundles over the Fano threefold ${\mathbbm{P}}(T_{{\mathbbm{P}^{2}}})$.

{\bf{Acknowledgments.}}  E. M. Correa is supported by S\~{a}o Paulo Research Foundation FAPESP grant 2022/10429-3.

\section{Ricci form of canonical connections} 

We start by recalling some basic facts related to Hermitian manifolds. Given a compact Hermitian manifold $(M,g,J)$, we have an associated $2$-form $\Omega \in \Omega^{2}(M)$, called fundamental $2$-form, such that
\begin{center}
$\Omega(X,Y) = g(JX,Y),$
\end{center}
$\forall X,Y \in TM$. From this, by considering the Levi-Civita connection $\nabla^{LC}$ associated with $g$, we can use the fundamental $2$-form to define a $1$-parametric family of connections $\nabla^{t} \colon \Gamma(TM) \times \Gamma(TM) \to \Gamma(TM)$, where $t$ is a free parameter, such that
\begin{equation}
g(\nabla^{t}_{X}(Y),Z) = g(\nabla^{LC}_{X}(Y),Z) + \frac{t-1}{4}({\rm{d}}^{c}\Omega)(X,Y,Z) + \frac{t+1}{4}({\rm{d}}^{c}\Omega)(X,JY,JZ), 
\end{equation}
$\forall X,Y,Z \in TM$, the connections $\nabla^{t}$ are called canonical connections of $(M,g,J)$ (a.k.a. $t$-Gauduchon connections), e.g. \cite{gauduchon11hermitian}. In order to describe the Ricci form associated with $\nabla^{t}$, we proceed as follows. Considering $\Omega:=g(J\otimes {\rm{id}})$, we have
\begin{equation}
\label{primitive}
({\rm{d}}\Omega)_{0} := {\rm{d}}\Omega - \frac{1}{n-1}{\rm{L}}(\Lambda({\rm{d}}\Omega)),
\end{equation}
where ${\rm{L}} = \Omega \wedge (-)$ is the Lefschetz operator and $\Lambda$ is its adjoint \cite{Wells}. Since $\Lambda({\rm{d}}\Omega)$ is primitive, it follows that $\Lambda({\rm{L}}(\Lambda({\rm{d}}\Omega))) = (n-1)\Lambda({\rm{d}}\Omega)$, thus $\Lambda(({\rm{d}}\Omega)_{0}) = 0$, i.e., the primitive part of ${\rm{d}}\Omega$ is given by Eq. (\ref{primitive}). In the above setting, we have the following definition.
\begin{definition}
\label{LeeDef}
The Lee $1$-form of a Hermitian manifold $(M,\Omega,J)$ of complex dimension $n$ is defined by
\begin{equation}
\label{Lee}
\theta:= \Lambda({\rm{d}}\Omega).
\end{equation}
\end{definition}
\begin{remark}
The Lee form $\theta$ associated with a Hermitian manifold plays an important role in the study of the classification of Hermitian structures, e.g. \cite{GrayHervella}.
\end{remark}
Now we consider the following result.
\begin{lemma}
 Let $(M,\Omega,J)$ be a Hermitian manifold of complex dimension $n$. Then 
\begin{equation}
\label{Lema1}
{\rm{d}}(\Omega^{n-1}) = \theta \wedge \Omega^{n-1}.
\end{equation}
\end{lemma}
\begin{remark}
The above lemma follows directly from the fact that
\begin{equation}
\Lambda(({\rm{d}}\Omega)_{0}) = 0 \iff ({\rm{d}}\Omega)_{0}\wedge \Omega^{n-2} = 0,
\end{equation}
see for instance \cite[Corollary 3.13]{Wells}.
\end{remark}
By considering the Hodge $\ast$-operator defined by $g$ and the associated codifferential $\delta := - \ast {\rm{d}} \ast $, from the decomposition ${\rm{d}}= \partial + \bar{\partial}$, we have $\delta := \partial^{\ast} + \bar{\partial}^{\ast}$, such that 
\begin{equation}
\partial^{\ast} := - \ast \bar{\partial} \ast, \ \ \ \bar{\partial}^{\ast}:= - \ast  \partial \ast.
\end{equation}
In the above context, since $\ast \Omega^{k} = \frac{k!}{(n-k)!}\Omega^{n-k}$, we have from Lemma \ref{Lema1} that
\begin{center}
$\displaystyle{\delta \Omega =  -\frac{1}{(n-1)!}\ast{\rm{d}}(\Omega^{n-1}) = -\frac{1}{(n-1)!}}\ast \big (\theta \wedge \Omega^{n-1}\big) = - \frac{1}{(n-1)!}\ast {\rm{L}}^{n-1}(\theta).$
\end{center}
Considering $J(\theta) = - \theta \circ J $, since $\Lambda(\theta) = 0$, it follows that $\ast {\rm{L}}^{n-1}(\theta) = (n-1)!J(\theta)$, e.g. \cite[Theorem 3.16]{Wells}. Therefore, it follows that  
\begin{equation}
\theta = J(\delta \Omega).
\end{equation}
Denoting by $R(\nabla^{t})$ the curvature tensor of $\nabla^{t}$ and by $\rho(\Omega,t)=\sqrt{-1}{\rm{tr}}(R(\nabla^{t}))$ its associated Ricci form, from \cite[Remark 5]{gauduchon11hermitian} we have that
\begin{equation}
{\rm{tr}}(R(\nabla^{t})) = {\rm{tr}}(R(\nabla^{u})) + \frac{\sqrt{-1}}{2}(t-u){\rm{d}}J(\theta),
\end{equation}
for all $t,u \in \mathbbm{R}$. Hence, we obtain the following
\begin{equation}
\rho(\Omega,t) = \rho(\Omega,u) + \frac{t-u}{2}{\rm{d}} \delta \Omega.
\end{equation}
Considering $\rho_{1}(\Omega,t) = \frac{1}{2}(\rho(\Omega,t) + J(\rho(\Omega,t)))$, i.e., the $(1,1)$ component of $\rho(\Omega,t)$, since 
\begin{enumerate}
\item $\rho(\Omega,1) = \sqrt{-1}{\rm{tr}}(R(\nabla^{Ch})) = \rho_{1}(\Omega,1),$
\item ${\rm{d}} \delta \Omega + J({\rm{d}} \delta \Omega) = 2 ( \partial \partial^{\ast} \Omega + \bar{\partial}\bar{\partial}^{\ast}\Omega),$
\end{enumerate}
we conclude that 
\begin{equation}
\rho_{1}(\Omega,t) =  \rho_{1}(\Omega,1)+ \frac{t-1}{2}\big ( \partial \partial^{\ast} \Omega + \bar{\partial}\bar{\partial}^{\ast}\Omega\big).
\end{equation}
\begin{definition}
Let $(M,g,J)$ be a compact Hermitian manifold and $\nabla^{t}$ the associated canonical connection, for some $t \in \mathbbm{R}$. We say that $(M,g,J)$ is $t$-Gauduchon Ricci-flat if $\rho_{1}(\Omega,t) = 0$.
\end{definition}
In the above setting:
\begin{enumerate}
\item if $\rho_{1}(\Omega,t) = 0$ for $t = 1 \Rightarrow$ $(M,g,J)$ is {\textit{Chern Ricci-flat}},
\item if $\rho_{1}(\Omega,t) = 0$ for $t = 0 \Rightarrow$ $(M,g,J)$ is {\textit{Lichnerowicz Ricci-flat}},
\item if $\rho_{1}(\Omega,t) = 0$ for $t = -1 \Rightarrow$ $(M,g,J)$ is {\textit{Strominger-Bismut Ricci-flat}}.
\end{enumerate}
Given a Hermitian manifold $(M,\Omega,J)$, the condition ${\rm{d}}\Omega^{k} = 0$, for some $1 \leq k \leq n-2$, $n = \dim_{\mathbbm{C}}(M)$, implies that ${\rm{d}}\Omega = 0$. In this setting, we shall also consider the following class of Hermitian metrics.

\begin{definition}
A Hermitian manifold $(M,\Omega,J)$ of complex dimension $n$ is said to be balanced if ${\rm{d}}\Omega^{n-1} = 0$.
\end{definition}
\begin{remark}
Notice that, since $\ast \Omega^{k} = \frac{k!}{(n-k)!}\Omega^{n-k}$, it follows that $(M,\Omega,J)$ is balanced if, and only if, $\Omega$ is co-closed (i.e, $\delta \Omega = 0$).
\end{remark}

\section{Constructions on principal torus bundles}

In this section, we highlight some aspects of gauge theory underlying the ideas introduced in \cite{GRANTCHAROV} and \cite{poddar2018group}. The main purpose is to present some sufficient conditions, in terms of primitive $(1,1)$-classes, to the existence of $t$-Gauduchon Ricci-flat metrics and balanced Hermitian metrics on principal torus bundles over K\"{a}hler-Einstein Fano manifolds. Let $M$ be the total space of a principal $T^{2r}$-bundle over a compact Hermitian manifold $B$ with characteristic classes of $(1,1)$-type. By choosing a connection 
\begin{center}
$\Theta = \begin{pmatrix} \sqrt{-1}\Theta_{1} & \cdots & 0 \\
 \vdots & \ddots & \vdots \\
 0 & \cdots & \sqrt{-1}\Theta_{2r}\end{pmatrix} \in \Omega^{1}(M;{\text{Lie}}(T^{2r})),$
\end{center}
we have ${\rm{d}}\Theta_{j} = \pi^{\ast}{\bf{\psi}}_{j}$, such that $\psi_{j} \in \Omega^{1,1}(B)$, $\forall j = 1,\ldots,2r$. From this, we can construct a complex structure $\mathbbm{J}$ on $M$ by using the horizontal lift of the complex structure defined on the base to $\ker(\Theta)$ (horizontal space), since the vertical space is identified with the tangent space of an even-dimensional torus, we can set
\begin{equation}
\mathbbm{J}(\Theta_{2k-1}) = -\Theta_{2k}, 1 \leq k \leq r.
\end{equation}
Thus, we have a well-defined almost complex structure $\mathbbm{J} \in {\text{End}}(TM)$. It is straightforward to verify that $\mathbbm{J}$ is in fact integrable, see \cite[Lemma 1]{GRANTCHAROV}. 

By considering a Hermitian metric $g_{B}$ on the base manifold $B$, we can use the connection $\Theta$ described above to define a Hermitian metric on $(M,\mathbbm{J})$. In fact, we can set
\begin{equation}
\label{metrictotal}
g_{M} = \pi^{\ast}(g_{B}) + \frac{1}{2}{\rm{tr}} \big ( \Theta \odot \Theta\big ) .
\end{equation}
here we consider $a \odot b = a \otimes b + b \otimes a$, $\forall a,b \in \Omega^{1}(M)$. Since $\mathbbm{J}(\Theta_{2k-1}) = -\Theta_{2k}$, $\forall k = 1,\dots,r$, it follows that the fundamental 2-form $\Omega_{M} = g_{M}(\mathbbm{J} \otimes {\rm{id}})$ is given by
\begin{equation}
\Omega_{M} = \pi^{\ast} (\omega_{B}) + \frac{1}{2} {\rm{tr}} \big ( \Theta \wedge \mathbbm{J} \Theta\big ),
\end{equation}
where $\omega_{B}$ is a the fundamental 2-form of $B$, here we consider $a \wedge b = a \otimes b - b \otimes a$, $\forall a,b \in \Omega^{1}(M)$. From \cite[Lemma 2]{GRANTCHAROV}, we have the following
\begin{equation}
\label{coderivativekahlerform}
\delta \Omega_{M} = \pi^{\ast} (\delta \omega_{B}) + \sum_{j = 1}^{2r} \pi^{\ast} \big (\Lambda_{\omega_{B}}(\psi_{j}) \big ) \Theta_{j},
\end{equation}
where $\Lambda_{\omega_{B}}$ is the dual of the Lefschetz operator ${\rm{L}}_{\omega_{B}} = \omega_{B} \wedge (-)$. Let us suppose from now on that $(B,\omega_{B},J)$ is a compact K\"{a}hler manifold. In this case, since $\delta \omega_{B} = 0$, it follows that 
\begin{equation}
\label{Balancedcondition}
(M,\Omega_{M},\mathbbm{J}) \ {\text{is balanced}} \ \ \iff \ \ \Lambda_{\omega_{B}}(\psi_{j}) = 0, \ \forall j=1,\ldots,2r.
\end{equation}
In particular, the condition $\delta \omega_{B} = 0$ also implies that  
\begin{equation}
\theta = \mathbbm{J}(\delta \Omega_{M}) = \sum_{j = 1}^{r} \Big ( \pi^{\ast} \big (\Lambda_{\omega_{B}}(\psi_{2j}) \big ) \Theta_{2j-1} - \pi^{\ast} \big (\Lambda_{\omega_{B}}(\psi_{2j-1}) \big ) \Theta_{2j}\Big ).
\end{equation}
Further, combining Eq. (\ref{coderivativekahlerform}) with Eq. (\ref{differenceRicci}), we obtain the following description for the Ricci curvature $\rho(\Omega_{M},t)$ associated with the canonical connection $\nabla^{t}$: 
\begin{equation}
\label{tgauduchonricvi}
\rho(\Omega_{M},t) = \pi^{\ast} \Big ( \rho_{B} + \frac{t-1}{2} \sum_{j = 1}^{2r} \Lambda_{\omega_{B}}(\psi_{j})\psi_{j}\Big ),
\end{equation}
where $\rho_{B}$ is the Ricci form of the Chern connection associated with $(B,\omega_{B})$, see for instance \cite[Proposition 5]{GRANTCHAROV}. Therefore, by construction, we have that $\rho(\Omega_{M},t)$ is of $(1,1)$-type, $\forall t \in \mathbbm{R}$, i.e.,
\begin{equation}
\rho(\Omega_{M},t) = \rho_{1}(\Omega_{M},t), \ \ \forall t \in \mathbbm{R}.
\end{equation}
From above, we see that the $t$-Gauduchon Ricci-flat condition on $(M,\Omega_{M},\mathbbm{J})$ is equivalent to the following equation on $(B,\omega_{B},J)$
\begin{equation}
\rho_{B} = - \frac{t-1}{2} \sum_{j = 1}^{2r} \Lambda_{\omega_{B}}(\psi_{j})\psi_{j}.
\end{equation}
Now we observe that, if $\Theta$ is a Yang-Mills connection, namely,
\begin{equation}
{\rm{d}}_{\Theta} \ast \Theta = 0,
\end{equation}
where ${\rm{d}}_{\Theta}$ is the covariant exterior derivative induced by $\Theta$, e.g. \cite{atiyah2005geometry}, \cite{rudolph2017differential}, since $T^{2r}$ is abelian, it follows that 
\begin{equation}
{\rm{d}}_{\Theta} \ast \Theta = 0 \iff {\rm{d}} \ast \psi_{j} = 0, \ \ \forall j = 1,\ldots,2r, \Rightarrow \Lambda_{\omega_{B}}(\psi_{j}) = {\rm{cte}},  \ \ \forall j = 1,\ldots,2r.
\end{equation}
Hence, if we suppose that\footnote{Notice that, since $[\frac{\psi_{1}}{2\pi}] \in H^{2}(B,\mathbbm{Z})$, the condition $\psi_{1} = \omega_{B}$ implies that $B$ is projective \cite{kodaira1954kahler}.} $\psi_{1} = \omega_{B}$, we obtain the following
\begin{equation}
\label{GRF}
\begin{cases} \displaystyle \rho_{B} = \frac{\dim_{\mathbbm{C}}(B)(1-t)}{2} \omega_{B}, \\
\Lambda_{\omega_{B}}(\psi_{j}) = 0, \forall j = 2,\ldots,2r .
\end{cases} \Longrightarrow \rho_{1}(\Omega_{M},t) = 0.
\end{equation}
As we can see, from Eq. (\ref{Balancedcondition}) and Eq. (\ref{GRF}), if we start with a compact K\"{a}hler-Einstein manifold $(M,\omega_{B},J)$ with positive scalar curvature, and with Picard number $\varrho(M) > 1$, we can construct examples of balanced Hermitian metrics and $t$-Gauduchon Ricci-flat metrics (for all $t < 1$) from elements of 
\begin{equation}
H^{1,1}_{\omega_{B}}(M,\mathbbm{Z})_{{\text{prim}}} := \ker \Big (\Lambda_{\omega_{B}} \colon H^{2}(B,\mathbbm{R}) \to \mathbbm{R} \Big ) \cap H^{1,1}(B,\mathbbm{Z}),
\end{equation}
i.e., from primitive integral $(1,1)$-classes. From now on, we will restrict our attention to the class of complex flag varieties. This class of varieties satisfies all the aforementioned requirements and we can take advantage of the underlying framework in Lie theory to construct explicit solutions to the system of PDE's given in Eq. (\ref{GRF}).
\begin{remark}
In the above setting, if the Picard number of $(M,\omega_{B},J)$ is $1$, then $H^{1,1}(B,\mathbbm{Z})$ is generated by $c_{1}(B)$ and $H^{1,1}_{\omega_{B}}(M,\mathbbm{Z})_{{\text{prim}}} = \{0\}$, since 
\begin{center}
$\Lambda_{\omega_{B}}(\rho_{B}) = 2{\rm{scal}}(g) > 0$. 
\end{center}
It is worth to point out that the above construction could not be extended to include K\"{a}hler-Einstein manifolds with negative scalar curvature as base manifolds, see for instance \cite[Page 22]{GRANTCHAROV}.
\end{remark}

\section{Generalities on flag varieties}\label{generalities}
In this section, we review some basic generalities about flag varieties. For more details on the subject presented in this section, we suggest \cite{Akhiezer}, \cite{Flagvarieties}, \cite{HumphreysLAG}, \cite{BorelRemmert}.
\subsection{The Picard group of flag varieties}
\label{subsec3.1}
Let $G^{\mathbbm{C}}$ be a connected, simply connected, and complex Lie group with simple Lie algebra $\mathfrak{g}^{\mathbbm{C}}$. By fixing a Cartan subalgebra $\mathfrak{h}$ and a simple root system $\Delta \subset \mathfrak{h}^{\ast}$, we have a decomposition of $\mathfrak{g}^{\mathbbm{C}}$ given by
\begin{center}
$\mathfrak{g}^{\mathbbm{C}} = \mathfrak{n}^{-} \oplus \mathfrak{h} \oplus \mathfrak{n}^{+}$, 
\end{center}
where $\mathfrak{n}^{-} = \sum_{\alpha \in \Phi^{-}}\mathfrak{g}_{\alpha}$ and $\mathfrak{n}^{+} = \sum_{\alpha \in \Phi^{+}}\mathfrak{g}_{\alpha}$, here we denote by $\Phi = \Phi^{+} \cup \Phi^{-}$ the root system associated with the simple root system $\Delta \subset \mathfrak{h}^{\ast}$. Let us denote by $\kappa$ the Cartan-Killing form of $\mathfrak{g}^{\mathbbm{C}}$. From this, for every  $\alpha \in \Phi^{+}$, we have $h_{\alpha} \in \mathfrak{h}$, such  that $\alpha = \kappa(\cdot,h_{\alpha})$, and we can choose $x_{\alpha} \in \mathfrak{g}_{\alpha}$ and $y_{-\alpha} \in \mathfrak{g}_{-\alpha}$, such that $[x_{\alpha},y_{-\alpha}] = h_{\alpha}$. From these data, we can define a Borel subalgebra\footnote{A maximal solvable subalgebra of $\mathfrak{g}^{\mathbbm{C}}$.} by setting $\mathfrak{b} = \mathfrak{h} \oplus \mathfrak{n}^{+}$. 

\begin{remark}
In the above setting, $\forall \phi \in \mathfrak{h}^{\ast}$, we also denote $\langle \phi, \alpha \rangle = \phi(h_{\alpha})$, $\forall \alpha \in \Phi^{+}$.
\end{remark}

Now we consider the following result (see for instance \cite{Flagvarieties}, \cite{HumphreysLAG}):
\begin{theorem}
Any two Borel subgroups are conjugate.
\end{theorem}
From the result above, given a Borel subgroup $B \subset G^{\mathbbm{C}}$, up to conjugation, we can always suppose that $B = \exp(\mathfrak{b})$. In this setting, given a parabolic Lie subgroup\footnote{A Lie subgroup which contains some Borel subgroup.} $P \subset G^{\mathbbm{C}}$, without loss of generality, we can suppose that
\begin{center}
$P  = P_{I}$, \ for some \ $I \subset \Delta$,
\end{center}
where $P_{I} \subset G^{\mathbbm{C}}$ is the parabolic subgroup which integrates the Lie subalgebra 
\begin{center}

$\mathfrak{p}_{I} = \mathfrak{n}^{+} \oplus \mathfrak{h} \oplus \mathfrak{n}(I)^{-}$, \ with \ $\mathfrak{n}(I)^{-} = \displaystyle \sum_{\alpha \in \langle I \rangle^{-}} \mathfrak{g}_{\alpha}$. 

\end{center}
By definition, we have that $P_{I} = N_{G^{\mathbbm{C}}}(\mathfrak{p}_{I})$, where $N_{G^{\mathbbm{C}}}(\mathfrak{p}_{I})$ is the normalizer in  $G^{\mathbbm{C}}$ of $\mathfrak{p}_{I} \subset \mathfrak{g}^{\mathbbm{C}}$, see for instance \cite[\S 3.1]{Akhiezer}. In what follows, it will be useful for us to consider the following basic chain of Lie subgroups

\begin{center}

$T^{\mathbbm{C}} \subset B \subset P \subset G^{\mathbbm{C}}$.

\end{center}
For each element in the aforementioned chain of Lie subgroups we have the following characterization: 

\begin{itemize}

\item $T^{\mathbbm{C}} = \exp(\mathfrak{h})$;  \ \ (complex torus)

\item $B = N^{+}T^{\mathbbm{C}}$, where $N^{+} = \exp(\mathfrak{n}^{+})$; \ \ (Borel subgroup)

\item $P = P_{I} = N_{G^{\mathbbm{C}}}(\mathfrak{p}_{I})$, for some $I \subset \Delta \subset \mathfrak{h}^{\ast}$. \ \ (parabolic subgroup)

\end{itemize}
Now let us recall some basic facts about the representation theory of $\mathfrak{g}^{\mathbbm{C}}$, a detailed exposition on the subject can be found in \cite{Humphreys}. For every $\alpha \in \Phi$, we set 
$$\alpha^{\vee} := \frac{2}{\langle \alpha, \alpha \rangle}\alpha.$$ 
The fundamental weights $\{\varpi_{\alpha} \ | \ \alpha \in \Delta\} \subset \mathfrak{h}^{\ast}$ of $(\mathfrak{g}^{\mathbbm{C}},\mathfrak{h})$ are defined by requiring that $\langle \varpi_{\alpha}, \beta^{\vee} \rangle= \delta_{\alpha \beta}$, $\forall \alpha, \beta \in \Delta$. We denote by 
$$\Lambda^{+} = \bigoplus_{\alpha \in \Delta}\mathbbm{Z}_{\geq 0}\varpi_{\alpha},$$ 
the set of integral dominant weights of $\mathfrak{g}^{\mathbbm{C}}$. Let $V$ be an arbitrary finite dimensional $\mathfrak{g}^{\mathbbm{C}}$-module. By considering its weight space decomposition
\begin{center}
$\displaystyle{V = \bigoplus_{\mu \in \Phi(V)}V_{\mu}},$ \ \ \ \ 
\end{center}
such that $V_{\mu} = \{v \in V \ | \ h \cdot v = \mu(h)v, \ \forall h \in \mathfrak{h}\} \neq \{0\}$, $\forall \mu \in \Phi(V) \subset \mathfrak{h}^{\ast}$, we have the following definition.

\begin{definition}
A highest weight vector (of weight $\lambda$) in a $\mathfrak{g}^{\mathbbm{C}}$-module $V$ is a non-zero vector $v_{\lambda}^{+} \in V_{\lambda}$, such that 
\begin{center}
$x \cdot v_{\lambda}^{+} = 0$, \ \ \ \ \ ($\forall x \in \mathfrak{n}^{+}$).
\end{center}
A weight $\lambda \in \Phi(V)$ associated with a highest weight vector is called highest weight of $V$.
\end{definition}

From above, we consider the following standard results (e.g. \cite{Humphreys}):
\begin{enumerate}

\item[(A)] Every finite dimensional irreducible $\mathfrak{g}^{\mathbbm{C}}$-module $V$ admits a highest weight vector $v_{\lambda}^{+}$. Moreover, $v_{\lambda}^{+}$ is the unique highest weight vector of $V$, up to non-zero scalar multiples.

\item[(B)] Let $V$ and $W$ be finite dimensional irreducible $\mathfrak{g}^{\mathbbm{C}}$-modules with highest weight $\lambda \in \mathfrak{h}^{\ast}$. Then, $V$ and $W$ are isomorphic. We will denote by $V(\lambda)$ a finite dimensional irreducible $\mathfrak{g}^{\mathbbm{C}}$-module with highest weight $\lambda \in \mathfrak{h}^{\ast}$.

\item[(C)] In the above setting, the following hold:
 
\begin{itemize}
\item[(C1)] If $V$ is a finite dimensional irreducible $\mathfrak{g}^{\mathbbm{C}}$-module with highest weight $\lambda \in \mathfrak{h}^{\ast}$, then $\lambda \in \Lambda^{+}$.

\item[(C2)] If $\lambda \in \Lambda^{+}$, then there exists a finite dimensional irreducible $\mathfrak{g}^{\mathbbm{C}}$-module $V$, such that $V = V(\lambda)$. 
\end{itemize}

\end{enumerate}
From item (C), it follows that the map $\lambda \mapsto V(\lambda)$ induces an one-to-one correspondence between $\Lambda^{+}$ and the isomorphism classes of finite dimensional irreducible $\mathfrak{g}^{\mathbbm{C}}$-modules.

\begin{remark} In what follows, it will be useful also to consider the following facts:
\begin{enumerate}
\item[(i)] For all $\lambda \in \Lambda^{+}$, we have $V(\lambda) = \mathfrak{U}(\mathfrak{g}^{\mathbbm{C}}) \cdot v_{\lambda}^{+}$, where $\mathfrak{U}(\mathfrak{g}^{\mathbbm{C}})$ is the universal enveloping algebra of $\mathfrak{g}^{\mathbbm{C}}$;
\item[(ii)] The fundamental representations are defined by $V(\varpi_{\alpha})$, $\alpha \in \Delta$; 
\item[(iii)] For all $\lambda \in \Lambda^{+}$, we have the following equivalence of induced irreducible representations
\begin{center}
$\varrho \colon G^{\mathbbm{C}} \to {\rm{GL}}(V(\lambda))$ \ $\Longleftrightarrow$ \ $\varrho_{\ast} \colon \mathfrak{g}^{\mathbbm{C}} \to \mathfrak{gl}(V(\lambda))$,
\end{center}
such that $\varrho(\exp(x)) = \exp(\varrho_{\ast}x)$, $\forall x \in \mathfrak{g}^{\mathbbm{C}}$, notice that $G^{\mathbbm{C}} = \langle \exp(\mathfrak{g}^{\mathbbm{C}}) \rangle$.
\end{enumerate}
\end{remark}
Given a representation $\varrho \colon G^{\mathbbm{C}} \to {\rm{GL}}(V(\lambda))$, for the sake of simplicity, we shall denote $\varrho(g)v = gv$, for all $g \in G^{\mathbbm{C}}$, and all $v \in V(\lambda)$. Let $G \subset G^{\mathbbm{C}}$ be a compact real form for $G^{\mathbbm{C}}$. Given a complex flag variety $X_{P} = G^{\mathbbm{C}}/P$, regarding $X_{P}$ as a homogeneous $G$-space, that is, $X_{P} = G/G\cap P$, the following theorem allows us to describe all $G$-invariant K\"{a}hler structures on $X_{P}$ through elements of representation theory.
\begin{theorem}[Azad-Biswas, \cite{AZAD}]
\label{AZADBISWAS}
Let $\omega \in \Omega^{1,1}(X_{P})^{G}$ be a closed invariant real $(1,1)$-form, then we have

\begin{center}

$\pi^{\ast}\omega = \sqrt{-1}\partial \overline{\partial}\varphi$,

\end{center}
where $\pi \colon G^{\mathbbm{C}} \to X_{P}$ is the natural projection, and $\varphi \colon G^{\mathbbm{C}} \to \mathbbm{R}$ is given by 
\begin{center}
$\varphi(g) = \displaystyle \sum_{\alpha \in \Delta \backslash I}c_{\alpha}\log \big (||gv_{\varpi_{\alpha}}^{+}|| \big )$, \ \ \ \ $(\forall g \in G^\mathbbm{C})$
\end{center}
with $c_{\alpha} \in \mathbbm{R}$, $\forall \alpha \in \Delta \backslash I$. Conversely, every function $\varphi$ as above defines a closed invariant real $(1,1)$-form $\omega_{\varphi} \in \Omega^{1,1}(X_{P})^{G}$. Moreover, $\omega_{\varphi}$ defines a $G$-invariant K\"{a}hler form on $X_{P}$ if and only if $c_{\alpha} > 0$,  $\forall \alpha \in \Delta \backslash I$.
\end{theorem}

\begin{remark}
\label{innerproduct}
It is worth pointing out that the norm $|| \cdot ||$ considered in the above theorem is a norm induced from some fixed $G$-invariant inner product $\langle \cdot, \cdot \rangle_{\alpha}$ on $V(\varpi_{\alpha})$, $\forall \alpha \in \Delta \backslash I$. 
\end{remark}

\begin{remark}
An important consequence of Theorem \ref{AZADBISWAS} is that it allows us to describe the local K\"{a}hler potential for any homogeneous K\"{a}hler metric in a quite concrete way, for some examples of explicit computations, we suggest \cite{CorreaGrama}, \cite{Correa}.
\end{remark}

By means of the above theorem we can describe the unique $G$-invariant representative of each integral class in $H^{2}(X_{P},\mathbbm{Z})$. In fact, consider the holomorphic $P$-principal bundle $P \hookrightarrow G^{\mathbbm{C}} \to X_{P}$. By choosing a trivializing open covering $X_{P} = \bigcup_{i \in J}U_{i}$, in terms of $\check{C}$ech cocycles we can write 
\begin{center}
$G^{\mathbbm{C}} = \Big \{(U_{i})_{i \in J}, \psi_{ij} \colon U_{i} \cap U_{j} \to P \Big \}$.
\end{center}
Given $\varpi_{\alpha} \in \Lambda^{+}$, we consider the induced character $\vartheta_{\varpi_{\alpha}} \in {\text{Hom}}(T^{\mathbbm{C}},\mathbbm{C}^{\times})$, such that $({\rm{d}}\vartheta_{\varpi_{\alpha}})_{e} = \varpi_{\alpha}$. Since $P = P_{I}$, we have the decomposition  
\begin{equation}
\label{parabolicdecomposition}
P_{I} = \big[P_{I},P_{I} \big]T(\Delta \backslash I)^{\mathbbm{C}}, \ \  {\text{such that }} \ \ T(\Delta \backslash I)^{\mathbbm{C}} = \exp \Big \{ \displaystyle \sum_{\alpha \in  \Delta \backslash I}a_{\alpha}h_{\alpha} \ \Big | \ a_{\alpha} \in \mathbbm{C} \Big \},
\end{equation}
e.g. \cite[\S 3]{Akhiezer}, so we can consider $\vartheta_{\varpi_{\alpha}} \in {\text{Hom}}(P,\mathbbm{C}^{\times})$ as being the trivial extension of $\vartheta_{\varpi_{\alpha}}|_{T(\Delta \backslash I)^{\mathbbm{C}}}$ to $P$. From the homomorphism $\vartheta_{\varpi_{\alpha}} \colon P \to \mathbbm{C}^{\times}$ one can equip $\mathbbm{C}$ with a structure of $P$-space, such that $pz = \vartheta_{\varpi_{\alpha}}(p)^{-1}z$, $\forall p \in P$, and $\forall z \in \mathbbm{C}$. Denoting by $\mathbbm{C}_{-\varpi_{\alpha}}$ this $P$-space, we can form an associated holomorphic line bundle $\mathscr{O}_{\alpha}(1) = G^{\mathbbm{C}} \times_{P}\mathbbm{C}_{-\varpi_{\alpha}}$, which can be described in terms of $\check{C}$ech cocycles by
\begin{equation}
\label{linecocycle}
\mathscr{O}_{\alpha}(1) = \Big \{(U_{i})_{i \in J},\vartheta_{\varpi_{\alpha}}^{-1} \circ \psi_{i j} \colon U_{i} \cap U_{j} \to \mathbbm{C}^{\times} \Big \},
\end{equation}
that is, $\mathscr{O}_{\alpha}(1) = \{g_{ij}\} \in \check{H}^{1}(X_{P},\mathcal{O}_{X_{P}}^{\ast})$, such that $g_{ij} = \vartheta_{\varpi_{\alpha}}^{-1} \circ \psi_{i j}$, $\forall i,j \in J$. 
\begin{remark}
\label{parabolicdec}
 In particular, if we take $\varpi_{\alpha} \in \Lambda^{+}$, such that $\alpha \in I$, since $\vartheta_{\varpi_{\alpha}}|_{T(\Delta \backslash I)^{\mathbbm{C}}}$ is trivial, it follows that $\mathscr{O}_{\alpha}(1)$ is trivial.
\end{remark}

\begin{remark}
Throughout this paper we shall use the following notation
\begin{equation}
\mathscr{O}_{\alpha}(k) := \mathscr{O}_{\alpha}(1)^{\otimes k},
\end{equation}
for every $k \in \mathbbm{Z}$ and every $\alpha \in \Delta \backslash I$. 
\end{remark}

Given $\mathscr{O}_{\alpha}(1) \in {\text{Pic}}(X_{P})$, such that $\alpha \in \Delta \backslash I$, as described above, if we consider an open covering $X_{P} = \bigcup_{i \in J} U_{i}$ which trivializes both $P \hookrightarrow G^{\mathbbm{C}} \to X_{P}$ and $ \mathscr{O}_{\alpha}(1) \to X_{P}$, by taking a collection of local sections $(s_{i})_{i \in J}$, such that $s_{i} \colon U_{i} \to G^{\mathbbm{C}}$, we can define $q_{i} \colon U_{i} \to \mathbbm{R}^{+}$, such that 
\begin{equation}
\label{functionshermitian}
q_{i} := \frac{1}{||s_{i}v_{\varpi_{\alpha}}^{+}||^{2}},
\end{equation}
for every $i \in J$. Since $s_{j} = s_{i}\psi_{ij}$ on $U_{i} \cap U_{j} \neq \emptyset$, and $pv_{\varpi_{\alpha}}^{+} = \vartheta_{\varpi_{\alpha}}(p)v_{\varpi_{\alpha}}^{+}$, for every $p \in P$, and every $\alpha \in \Delta \backslash I$, the collection of functions $(q_{i})_{i \in J}$ satisfy $q_{j} = |\vartheta_{\varpi_{\alpha}}^{-1} \circ \psi_{ij}|^{2}q_{i}$ on $U_{i} \cap U_{j} \neq \emptyset$. Hence, we obtain a collection of functions $(q_{i})_{i \in J}$ which satisfies the following relation on the overlaps $U_{i} \cap U_{j} \neq \emptyset$
\begin{equation}
\label{collectionofequ}
q_{j} = |g_{ij}|^{2}q_{i},
\end{equation}
such that $g_{ij} = \vartheta_{\varpi_{\alpha}}^{-1} \circ \psi_{i j}$, $\forall i,j \in J$. From this, we can define a Hermitian structure ${\bf{h}}$ on $\mathscr{O}_{\alpha}(1)$ by taking on each trivialization $\tau_{i} \colon U_{i} \times \mathbbm{C} \to \mathscr{O}_{\alpha}(1)$ the metric defined by
\begin{equation}
\label{hermitian}
{\bf{h}}(\tau_{i}(x,v),\tau_{i}(x,w)) = q_{i}(x) v\overline{w},
\end{equation}
for every $(x,v),(x,w) \in U_{i} \times \mathbbm{C}$. The Hermitian metric above induces a Chern connection $\nabla \colon \mathcal{A}^{0}(\mathscr{O}_{\alpha}(1)) \to \mathcal{A}^{1}(\mathscr{O}_{\alpha}(1))$, such that 
\begin{equation}
\nabla \sigma_{i} = \partial (\log q_{i} )\otimes \sigma_{i},
\end{equation}
where $\sigma_{i}(-) = \tau_{i}(-,1), \forall i \in J$. From above, the curvature of $\nabla$ is given by
\begin{equation}
\displaystyle F_{\nabla} \big |_{U_{i}} =   \partial \overline{\partial}\log \Big ( \big | \big | s_{i}v_{\varpi_{\alpha}}^{+}\big | \big |^{2} \Big).
\end{equation}
Therefore, by considering the closed $G$-invariant $(1,1)$-form ${\bf{\Omega}}_{\alpha} \in \Omega^{1,1}(X_{P})^{G}$, which satisfies $\pi^{\ast}{\bf{\Omega}}_{\alpha} = \sqrt{-1}\partial \overline{\partial} \varphi_{\varpi_{\alpha}}$, where $\pi \colon G^{\mathbbm{C}} \to G^{\mathbbm{C}} / P = X_{P}$, and $\varphi_{\varpi_{\alpha}}(g) = \frac{1}{2\pi}\log||gv_{\varpi_{\alpha}}^{+}||^{2}$, $\forall g \in G^{\mathbbm{C}}$, we have 
\begin{equation}
{\bf{\Omega}}_{\alpha} |_{U_{i}} = (\pi \circ s_{i})^{\ast}{\bf{\Omega}}_{\alpha} = \frac{\sqrt{-1}}{2\pi}F_{\nabla} \Big |_{U_{i}},
\end{equation}
i.e., $c_{1}(\mathscr{O}_{\alpha}(1)) = [ {\bf{\Omega}}_{\alpha}]$, $\forall \alpha \in \Delta \backslash I$.

\begin{remark}
Given $I \subset \Delta$, we shall denote $\Phi_{I}^{\pm}:= \Phi^{\pm} \backslash \langle I \rangle^{\pm}$, such that $\langle I \rangle^{\pm} = \langle I \rangle \cap \Phi^{\pm}$.
\end{remark}

\begin{remark}
\label{bigcellcosntruction}
In order to perform some local computations we shall consider the open set $U^{-}(P) \subset X_{P}$ defined by the ``opposite" big cell in $X_{P}$. This open set is a distinguished coordinate neighbourhood $U^{-}(P) \subset X_{P}$ of ${\rm{o}} := eP \in X_{P}$ defined as follows
\begin{equation}
\label{bigcell}
 U^{-}(P) =  B^{-}{\rm{o}} = R_{u}(P_{I})^{-}{\rm{o}} \subset X_{P},  
\end{equation}
 where $B^{-} = \exp(\mathfrak{h} \oplus \mathfrak{n}^{-})$, and
 
 \begin{center}
 
 $R_{u}(P_{I})^{-} = \displaystyle \prod_{\alpha \in \Phi_{I}^{+}}N_{\alpha}^{-}$, \ \ (opposite unipotent radical)
 
 \end{center}
with $N_{\alpha}^{-} = \exp(\mathfrak{g}_{-\alpha})$, $\forall \alpha \in \Phi_{I}^{+}$, e.g. \cite[\S 3]{Lakshmibai2},\cite[\S 3.1]{Akhiezer}. It is worth mentioning that the opposite big cell defines a contractible open dense subset in $X_{P}$, thus the restriction of any vector bundle (principal bundle) over this open set is trivial.
\end{remark}

Consider now the following result.

\begin{lemma}
\label{funddynkinline}
Consider $\mathbbm{P}_{\beta}^{1} = \overline{\exp(\mathfrak{g}_{-\beta}){\rm{o}}} \subset X_{P}$, such that $\beta \in \Phi_{I}^{+}$. Then, 
\begin{equation}
\int_{\mathbbm{P}_{\beta}^{1}} {\bf{\Omega}}_{\alpha} = \langle \varpi_{\alpha}, \beta^{\vee}  \rangle, \ \forall \alpha \in \Delta \backslash I.
\end{equation}
\end{lemma}

A proof for the above result can be found in \cite{correa2023deformed}, see also \cite{FultonWoodward} and \cite{AZAD}. From the above lemma and Theorem \ref{AZADBISWAS}, we obtain the following fundamental result.

\begin{proposition}[\cite{correa2023deformed}]
\label{C8S8.2Sub8.2.3P8.2.6}
Let $X_{P}$ be a complex flag variety associated with some parabolic Lie subgroup $P = P_{I}$. Then, we have
\begin{equation}
\label{picardeq}
{\text{Pic}}(X_{P}) = H^{1,1}(X_{P},\mathbbm{Z}) = H^{2}(X_{P},\mathbbm{Z}) = \displaystyle \bigoplus_{\alpha \in \Delta \backslash I}\mathbbm{Z}[{\bf{\Omega}}_{\alpha} ].
\end{equation}
\end{proposition}


In the above setting, we consider the weights of $P = P_{I}$ as being  
\begin{center}
$\displaystyle \Lambda_{P} := \bigoplus_{\alpha \in \Delta \backslash I}\mathbbm{Z}\varpi_{\alpha}$. 
\end{center}
From this, the previous result provides $\Lambda_{P} \cong {\rm{Hom}}(P,\mathbbm{C}^{\times}) \cong {\rm{Pic}}(X_{P})$, such that 
\begin{enumerate}
\item$ \displaystyle \lambda = \sum_{\alpha \in \Delta \backslash I}k_{\alpha}\varpi_{\alpha} \mapsto \prod_{\alpha \in \Delta \backslash I} \vartheta_{\varpi_{\alpha}}^{k_{\alpha}} \mapsto \bigotimes_{\alpha \in \Delta \backslash I} \mathscr{O}_{\alpha}(k_{\alpha})$;
\item $ \displaystyle {\bf{E}} \mapsto \vartheta_{{\bf{E}}}: = \prod_{\alpha \in \Delta \backslash I} \vartheta_{\varpi_{\alpha}}^{\langle c_{1}({\bf{L}}),[\mathbbm{P}^{1}_{\alpha}] \rangle} \mapsto \lambda({\bf{E}}) := \sum_{\alpha \in \Delta \backslash I}\langle c_{1}({\bf{E}}),[\mathbbm{P}^{1}_{\alpha}] \rangle\varpi_{\alpha}$.
\end{enumerate}
Thus, $\forall {\bf{E}} \in {\rm{Pic}}(X_{P})$, we have $\lambda({\bf{E}}) \in \Lambda_{P}$. More generally, $\forall \xi \in H^{1,1}(X_{P},\mathbbm{R})$, we can attach $\lambda (\xi) \in \Lambda_{P}\otimes \mathbbm{R}$, such that
\begin{equation}
\label{weightcohomology}
\lambda(\xi) := \sum_{\alpha \in \Delta \backslash I}\langle \xi,[\mathbbm{P}^{1}_{\alpha}] \rangle\varpi_{\alpha}.
\end{equation}
From above, for every holomorphic vector bundle ${\bf{E}} \to X_{P}$, we define $\lambda({\bf{E}}) \in \Lambda_{P}$, such that 
\begin{equation}
\label{weightholomorphicvec}
\lambda({\bf{E}}) := \sum_{\alpha \in \Delta \backslash I} \langle c_{1}({\bf{E}}),[\mathbbm{P}_{\alpha}^{1}] \rangle \varpi_{\alpha},
\end{equation}
where $c_{1}({\bf{E}}) = c_{1}(\bigwedge^{r}{\bf{E}})$, such that $r = \rank({\bf{E}})$.

\begin{remark}[Harmonic 2-forms on $X_{P}$]Given any $G$-invariant Riemannian metric $g$ on $X_{P}$, let us denote by $\mathscr{H}^{2}(X_{P},g)$ the space of real harmonic 2-forms on $X_{P}$ with respect to $g$, and let us denote by $\mathscr{I}_{G}^{1,1}(X_{P})$ the space of closed invariant real $(1,1)$-forms. Combining the result of Proposition \ref{C8S8.2Sub8.2.3P8.2.6} with \cite[Lemma 3.1]{MR528871}, we obtain 
\begin{equation}
\mathscr{I}_{G}^{1,1}(X_{P}) = \mathscr{H}^{2}(X_{P},g). 
\end{equation}
Therefore, the closed $G$-invariant real $(1,1)$-forms described in Theorem \ref{AZADBISWAS} are harmonic with respect to any $G$-invariant Riemannian metric on $X_{P}$.
\end{remark}

\begin{remark}[K\"{a}hler cone of $X_{P}$]It follows from Eq. (\ref{picardeq}) and Theorem \ref{AZADBISWAS} that the K\"{a}hler cone of a complex flag variety $X_{P}$ is given explicitly by
\begin{equation}
\mathcal{K}(X_{P}) = \displaystyle \bigoplus_{\alpha \in \Delta \backslash I} \mathbbm{R}^{+}[ {\bf{\Omega}}_{\alpha}].
\end{equation}
\end{remark}

\begin{remark}[Cone of curves of $X_{P}$] It is worth observing that the cone of curves ${\rm{NE}}(X_{P})$ of a flag variety $X_{P}$ is generated by the rational curves $[\mathbbm{P}_{\alpha}^{1}] \in \pi_{2}(X_{P})$, $\alpha \in \Delta \backslash I$, see for instance \cite[\S 18.3]{Timashev} and references therein.
\end{remark}

\begin{proposition}
\label{eigenvalueatorigin}
Let $X_{P}$ be a flag variety and let $\omega_{0}$ be a $G$-invariant K\"{a}hler metric on $X_{P}$. Then, for every closed $G$-invariant real $(1,1)$-form $\psi$, the eigenvalues of the endomorphism $\omega_{0}^{-1} \circ \psi$ are given by 
\begin{equation}
\label{eigenvalues}
{\bf{q}}_{\beta}(\omega_{0}^{-1} \circ \psi) = \frac{ \langle \lambda([\psi]), \beta^{\vee} \rangle}{\langle \lambda([\omega_{0}]), \beta^{\vee} \rangle}, \ \ \beta \in \Phi_{I}^{+}.
\end{equation}
\end{proposition}

A proof for the above result can be found in \cite{correa2023deformed}.

\begin{remark}
\label{primitivecalc}
In the setting of the last proposition, since $n \psi\wedge \omega_{0}^{n-1} = \Lambda_{\omega_{0}}(\psi)\omega_{0}^{n}$, such that $n=\dim_{\mathbbm{C}}(X_{P})$, and $\Lambda_{\omega_{0}}(\psi)={\rm{tr}}(\omega_{0}^{-1} \circ \psi)$, it follows that
\begin{equation}
\label{contraction}
\Lambda_{\omega_{0}}({\bf{\Omega}}_{\alpha})=\sum_{\beta \in \Phi_{I}^{+}} \frac{\langle \varpi_{\alpha}, \beta^{\vee} \rangle}{\langle \lambda([\omega_{0}]), \beta^{\vee}\rangle},
\end{equation}
for every $\alpha \in \Delta \backslash I$. In particular, for every ${\bf{E}} \in {\rm{Pic}}(X_{P})$, we have a Hermitian structure ${\bf{h}}$ on ${\bf{E}}$, such that the curvature $F_{\nabla}$ of the Chern connection $\nabla \myeq {\rm{d}} + \partial \log ({\bf{h}})$, satisfies 
\begin{equation}
\frac{\sqrt{-1}}{2\pi} \Lambda_{\omega_{0}}(F_{\nabla}) = \sum_{\beta \in \Phi_{I}^{+} } \frac{\langle \lambda({\bf{E}}), \beta^{\vee} \rangle}{\langle \lambda([\omega_{0}]), \beta^{\vee}\rangle}.
\end{equation}
From this, we have that $\nabla$ is a Hermitian-Yang-Mills (HYM) connection. Notice that 
\begin{equation}
c_{1}({\bf{E}}) = \sum_{\alpha \in \Delta \backslash I}\langle \lambda({\bf{E}}), \alpha^{\vee} \rangle [{\bf{\Omega}}_{\alpha}],
\end{equation}
for every ${\bf{E}} \in {\rm{Pic}}(X_{P})$, i.e., the curvature of the HYM connection $\nabla$ on ${\bf{E}}$ coincides with the $G$-invariant representative of $c_{1}({\bf{E}})$. 
\end{remark}

\subsection{The first Chern class of flag varieties} In this subsection, we shall review some basic facts related with the Ricci form of $G$-invariant K\"{a}hler metrics on flag varieties. Let $X_{P}$ be a complex flag variety associated with some parabolic Lie subgroup $P = P_{I} \subset G^{\mathbbm{C}}$. By considering the identification $T_{{\rm{o}}}^{1,0}X_{P} \cong \mathfrak{m} \subset \mathfrak{g}^{\mathbbm{C}}$, such that 

\begin{center}
$\mathfrak{m} = \displaystyle \sum_{\alpha \in \Phi_{I}^{-}} \mathfrak{g}_{\alpha}$,
\end{center}
 we can realize $T^{1,0}X_{P}$ as being a holomoprphic vector bundle, associated with the holomorphic principal $P$-bundle $P \hookrightarrow G^{\mathbbm{C}} \to X_{P}$, such that 

\begin{center}

$T^{1,0}X_{P} = \Big \{(U_{i})_{i \in J}, \underline{{\rm{Ad}}}\circ \psi_{i j} \colon U_{i} \cap U_{j} \to {\rm{GL}}(\mathfrak{m}) \Big \}$,

\end{center}
where $\underline{{\rm{Ad}}} \colon P \to {\rm{GL}}(\mathfrak{m})$ is the isotropy representation. From this, we obtain 
\begin{equation}
\label{canonicalbundleflag}
{\bf{K}}_{X_{P}}^{-1} = \det \big(T^{1,0}X_{P} \big) = \Big \{(U_{i})_{i \in J}, \det (\underline{{\rm{Ad}}}\circ \psi_{i j}) \colon U_{i} \cap U_{j} \to \mathbbm{C}^{\times} \Big \}.
\end{equation}
Since $P= [P,P] T(\Delta \backslash I)^{\mathbbm{C}}$, considering $\det \circ \underline{{\rm{Ad}}} \in {\text{Hom}}(T(\Delta \backslash I)^{\mathbbm{C}},\mathbbm{C}^{\times})$, we have 
\begin{equation}
\det \underline{{\rm{Ad}}}(\exp({\bf{t}})) = {\rm{e}}^{{\rm{tr}}({\rm{ad}}({\bf{t}})|_{\mathfrak{m}})} = {\rm{e}}^{- \langle \delta_{P},{\bf{t}}\rangle },
\end{equation}
$\forall {\bf{t}} \in {\rm{Lie}}(T(\Delta \backslash I)^{\mathbbm{C}})$, such that $\delta_{P} = \sum_{\alpha \in \Phi_{I}^{+} } \alpha$. Denoting $\vartheta_{\delta_{P}}^{-1} = \det \circ \underline{{\rm{Ad}}}$, it follows that 
\begin{equation}
\label{charactercanonical}
\vartheta_{\delta_{P}} = \displaystyle \prod_{\alpha \in \Delta \backslash I} \vartheta_{\varpi_{\alpha}}^{\langle \delta_{P},\alpha^{\vee} \rangle} \Longrightarrow {\bf{K}}_{X_{P}}^{-1} = \bigotimes_{\alpha \in \Delta \backslash I}\mathscr{O}_{\alpha}(\ell_{\alpha}),
\end{equation}
such that $\ell_{\alpha} = \langle \delta_{P}, \alpha^{\vee} \rangle, \forall \alpha \in \Delta \backslash I$. Notice that $\lambda({\bf{K}}_{X_{P}}^{-1}) = \delta_{P}$, see Eq. (\ref{weightholomorphicvec}). If we consider the invariant K\"{a}hler metric $\rho_{0} \in \Omega^{1,1}(X_{P})^{G}$ defined by
\begin{equation}
\label{riccinorm}
\rho_{0} = \sum_{\alpha \in \Delta \backslash I}2 \pi \langle \delta_{P}, \alpha^{\vee} \rangle {\bf{\Omega}}_{\alpha},
\end{equation}
it follows that
\begin{equation}
\label{ChernFlag}
c_{1}(X_{P}) = \Big [ \frac{\rho_{0}}{2\pi}\Big].
\end{equation}
By the uniqueness of $G$-invariant representative of $c_{1}(X_{P})$, we have 
\begin{center}
${\rm{Ric}}(\rho_{0}) = \rho_{0}$, 
\end{center}
i.e., $\rho_{0} \in \Omega^{1,1}(X_{P})^{G}$ defines a $G$-ivariant K\"{a}hler-Einstein metric on $X_{P}$ (cf. \cite{MATSUSHIMA}). 

\begin{remark}
\label{Fanoindex}
Observe that, from Proposition \ref{C8S8.2Sub8.2.3P8.2.6} and Eq. (\ref{riccinorm}), we have the Fano index of $X_{P}$ given by the greatest common divisor
\begin{equation}   
{\bf{I}}(X_{P}) = {\text{G.C.D.}} \Big (  \langle \delta_{P}, \alpha^{\vee} \rangle \ \Big | \ \alpha \in \Delta \backslash I \Big ),    
\end{equation}
here we suppose $P = P_{I} \subset G^{\mathbbm{C}}$, for some $I \subset \Delta$. Thus, $I(X_{P})$ can be completely determined by the Cartan matrix of $\mathfrak{g}^{\mathbbm{C}}$.
\end{remark}

\begin{remark}
Given any $G$-invariant K\"{a}hler metric $\omega$ on $X_{P}$, we have ${\rm{Ric}}(\omega) = \rho_{0}$. Thus, it follows that the smooth function $\frac{\det(\omega)}{\det(\rho_{0})}$ is constant. From this, we obtain
\begin{equation}
{\rm{Vol}}(X_{P},\omega) = \frac{1}{n!}\int_{X_{P}}\omega^{n}  =  \frac{\det(\rho_{0}^{-1} \circ \omega)}{n!} \int_{X_{P}}\rho_{0}^{n}.
\end{equation}
Since 
\begin{equation}
\det(\rho_{0}^{-1} \circ \omega) = \frac{1}{(2\pi)^{n}} \prod_{\beta \in \Phi_{I}^{+}}\frac{\langle \lambda([\omega]),\beta^{\vee} \rangle}{\langle \delta_{P},\beta^{\vee} \rangle} \ \ {\text{and}} \ \ \frac{1}{n!}\int_{X_{P}}c_{1}(X_{P})^{n} = \prod_{\beta \in \Phi_{I}^{+}} \frac{\langle \delta_{P},\beta^{\vee} \rangle}{\langle \varrho^{+},\beta^{\vee}\rangle}, 
\end{equation}
we conclude that\footnote{cf. \cite{AZAD}.} 
\begin{equation}
\label{VolKahler}
{\rm{Vol}}(X_{P},\omega) = \prod_{\beta \in \Phi_{I}^{+}} \frac{\langle \lambda([\omega]),\beta^{\vee} \rangle}{\langle \varrho^{+},\beta^{\vee} \rangle},
\end{equation}
where $\varrho^{+} = \frac{1}{2} \sum_{\alpha \in \Phi^{+}}\alpha$. Combining the above formula with the ideas introduced in Remark \ref{primitivecalc} we obtain the following expression for the degree of a holomorphic vector bundle ${\bf{E}} \to X_{P}$ with respect to some $G$-invariant K\"{a}hler metric $\omega$ on $X_{P}$:
\begin{equation}
\label{degreeVB}
\deg_{\omega}({\bf{E}}) = \int_{X_{P}}c_{1}({\bf{E}}) \wedge [\omega]^{n-1} = (n-1)!\Bigg [\sum_{\beta \in \Phi_{I}^{+} } \frac{\langle \lambda({\bf{E}}), \beta^{\vee} \rangle}{\langle \lambda([\omega]), \beta^{\vee}\rangle}\Bigg ] \Bigg [\prod_{\beta \in \Phi_{I}^{+}} \frac{\langle \lambda([\omega]),\beta^{\vee} \rangle}{\langle \varrho^{+},\beta^{\vee} \rangle} \Bigg ],
\end{equation}
such that $\lambda({\bf{E}}) \in \Lambda_{P}$, and $\lambda([\omega]) \in \Lambda_{P} \otimes \mathbbm{R}$. 
\end{remark}

\subsection{Primitive (1,1)-forms on flag varieties}
\label{primitivesplitting}
Let $X_{P}$ be a complex flag variety with Picard number $\varrho(X_{P})>1$. Considering ${\bf{E}} \in {\rm{Pic}}(X_{P})$, fixed some K\"{a}hler class $[\omega_{0}] \in H^{2}(X_{P},\mathbbm{F})$, such that $\mathbbm{F} = \mathbbm{Z}, \mathbbm{Q}$, or $\mathbbm{R}$, we have the $[\omega_{0}]$-degree of ${\bf{E}}$ given by
\begin{equation}
\deg_{\omega_{0}}({\bf{E}}):= \int_{X_{P}}c_{1}({\bf{E}}) \wedge [\omega_{0}]^{n-1}.
\end{equation}
In particular, considering $\mathscr{O}_{\alpha}(1)$, $\forall \alpha \in \Delta \backslash I$, since $c_{1}(\mathscr{O}_{\alpha}(1))$ is an integral class, $\forall \alpha \in \Delta \backslash I$, and $[\omega_{0}] \in H^{2}(X_{P},\mathbbm{F})$, it follows that 
\begin{equation}
\deg_{\omega_{0}}(\mathscr{O}_{\alpha}(1)) = \int_{X_{P}}c_{1}(\mathscr{O}_{\alpha}(1)) \wedge [\omega_{0}] \wedge [\omega_{0}]^{n-2} = \mathcal{Q}_{\omega_{0}}\big (c_{1}(\mathscr{O}_{\alpha}(1)),[\omega_{0}] \big ) \in \mathbbm{F}, 
\end{equation}
$\forall \alpha \in \Delta \backslash I$, where $\mathcal{Q}_{\omega_{0}} \colon H^{2}(X_{P},\mathbbm{F}) \times H^{2}(X_{P},\mathbbm{F}) \to \mathbbm{F}$ is the underlying Hodge-Riemann bilinear form (e.g. \cite{VoisinBook1}, \cite{peters2008mixed}). If we denote by $\psi \in c_{1}({\bf{E}})$ the unique $G$-invariant representative, we have that 
\begin{equation}
\label{degreecontraction}
\deg_{\omega_{0}}({\bf{E}}) = (n-1)!\Lambda_{\omega_{0}}(\psi){\rm{Vol}}(X_{P},\omega_{0}),
\end{equation}
such that ${\rm{Vol}}(X_{P},\omega_{0}) = \frac{1}{n!}\int_{X_{P}}\omega_{0}^{n}$. Thus, it follows that
\begin{equation}
\label{primitivedegree}
\deg_{\omega_{0}}({\bf{E}})= 0 \iff \Lambda_{\omega_{0}}(\psi) = 0,
\end{equation}
i.e., ${\bf{E}} \in {\rm{Pic}}(X_{P})$ has zero $[\omega_{0}]$-degree if, and only if, $c_{1}({\bf{E}})$ is a primitive class (w.r.t. $\omega_{0}$). From this, we consider the primitive submodule 
\begin{equation}
H^{2}_{\omega_{0}}(X_{P},\mathbbm{Z})_{{\text{prim}}} := \ker \Big (\Lambda_{\omega_{0}} \colon H^{2}(X_{P},\mathbbm{R}) \to \mathbbm{R} \Big ) \cap H^{2}(X_{P},\mathbbm{Z}).
\end{equation}
Defining
\begin{equation}
{\rm{Pic}}^{0}_{\omega_{0}}(X_{P}) = \Big \{ {\bf{E}} \in {\rm{Pic}}(X_{P}) \ \Big | \ \deg_{\omega_{0}}({\bf{E}}) = 0 \Big \},
\end{equation}
it follows from Eq. (\ref{primitivedegree}) that ${\rm{Pic}}^{0}_{\omega_{0}}(X_{P}) \cong H^{2}_{\omega_{0}}(X_{P},\mathbbm{Z})_{{\text{prim}}}$.
\begin{remark}[Case $\mathbbm{F} = \mathbbm{Z}$]
\label{integralcase}
In the particular case that $\mathbbm{F} = \mathbbm{Z}$, i.e., $(X_{P}, \omega_{0})$ is a Hodge manifold, we can construct an explicit $\mathbbm{Z}$-basis for $H^{2}_{\omega_{0}}(X_{P},\mathbbm{Z})_{{\text{prim}}}$. In fact, given $[\psi] \in H^{2}(X_{P},\mathbbm{Z})$, we have 
\begin{equation}
\Lambda_{\omega_{0}}([\psi]) = \sum_{\alpha \in \Delta \backslash I}\Lambda_{\omega_{0}}([{\bf{\Omega}}_{\alpha}]) x_{\alpha},
\end{equation}
such that $x_{\alpha} = \big \langle \lambda([\psi]), [\mathbbm{P}^{1}_{\alpha}] \big \rangle$, $\forall \alpha \in \Delta \backslash I$. Since 
\begin{equation}
\label{contractionvolume}
(n-1)!\Lambda_{\omega_{0}}([{\bf{\Omega}}_{\alpha}]){\rm{Vol}}(X_{P},\omega_{0}) = \mathcal{Q}_{\omega_{0}}({\bf{\Omega}}_{\alpha},\omega_{0}) \in \mathbbm{Z}, 
\end{equation}
$\forall \alpha \in \Delta \backslash I$, we have
\begin{equation}
\Lambda_{\omega_{0}}([\psi]) = 0 \iff \sum_{\alpha \in \Delta \backslash I}\frac{\mathcal{Q}_{\omega_{0}}({\bf{\Omega}}_{\alpha},\omega_{0})}{{\bf{\tau}}([\omega_{0}])}x_{\alpha}=0.
\end{equation}
where
\begin{equation}
\tau([\omega_{0}]) := {\rm{G.C.D.}} \Big ( \mathcal{Q}_{\omega_{0}}({\bf{\Omega}}_{\alpha},\omega_{0}) \ \Big | \ \alpha \in \Delta \backslash I \Big).
\end{equation}
Denoting $q_{\alpha}(\omega_{0}):= \frac{\mathcal{Q}_{\omega_{0}}({\bf{\Omega}}_{\alpha},\omega_{0})}{\tau([\omega_{0}])}$, $\forall \alpha \in \Delta \backslash I$, and taking some $\gamma \in \Delta \backslash I$, we obtain a $\mathbbm{Z}$-basis for $H^{2}_{\omega_{0}}(X_{P},\mathbbm{Z})_{{\text{prim}}}$ by setting 
\begin{equation}
\label{basiszeroslope}
{\bf{\xi}}_{\alpha} := -q_{\alpha}(\omega_{0})[{\bf{\Omega}}_{\gamma}] + q_{\gamma}(\omega_{0})[{\bf{\Omega}}_{\alpha}], 
\end{equation}
for all $\forall \alpha \in \Delta \backslash I$, such that $\alpha \neq \gamma$. From this, we have
\begin{equation}
H_{\omega_{0}}^{2}(X_{P},\mathbbm{Z})_{{\text{prim}}} = \bigoplus_{\substack{\alpha \in \Delta \backslash I, \alpha \neq \gamma}}\mathbbm{Z}\xi_{\alpha}.
\end{equation}
By construction, we obtain a $\mathcal{Q}_{\omega_{0}}$-orthogonal decomposition 
\begin{equation}
H^{2}(X_{P},\mathbbm{Z}) = \mathbbm{Z}[\omega_{0}] \oplus H_{\omega_{0}}^{2}(X_{P},\mathbbm{Z})_{{\text{prim}}}.
\end{equation}
From Eq. (\ref{basiszeroslope}), we can describe explicitly a set of generators for the kernel of the homomorphism $\deg_{\omega_{0}}\colon {\rm{Pic}}(X_{P}) \to \mathbbm{Z}$. In fact, considering
\begin{equation}
{\rm{Pic}}^{0}_{\omega_{0}}(X_{P}) = \Big \{ {\bf{E}} \in {\rm{Pic}}(X_{P}) \ \Big | \ \deg_{\omega_{0}}({\bf{E}}) = 0 \Big \},
\end{equation}
since ${\rm{Pic}}^{0}_{\omega_{0}}(X_{P}) \cong H^{2}_{\omega_{0}}(X_{P},\mathbbm{Z})_{{\text{prim}}}$, it follows from Eq. (\ref{basiszeroslope}) that 
\begin{equation} \mathscr{O}_{\gamma}(-q_{\alpha}(\omega_{0})) \otimes \mathscr{O}_{\alpha}(q_{\gamma}(\omega_{0})), \ \ \forall \alpha \in \Delta \backslash I, \alpha \neq \gamma,
\end{equation}
define a set of generators for ${\rm{Pic}}^{0}_{\omega_{0}}(X_{P})$. Moreover, if we consider the finitely generated subgroup ${\rm{H}}_{\omega_{0}}\subset {\text{Hom}}(P,\mathbbm{C}^{\times})$, such that
\begin{equation}
{\rm{H}}_{\omega_{0}} := \Big \langle \chi_{\varpi_{\gamma}}^{-q_{\alpha}(\omega_{0})} \chi_{\varpi_{\alpha}}^{q_{\gamma}(\omega_{0})} \ \Big | \ \alpha \in \Delta \backslash I, \alpha \neq \gamma \Big \rangle,
\end{equation}
recall that ${\text{Hom}}(P,\mathbbm{C}^{\times}) = {\text{Hom}}(T(\Delta \backslash I)^{\mathbbm{C}},\mathbbm{C}^{\times})$, and $({\rm{d}}\chi_{\varpi_{\alpha}})_{e} = \varpi_{\alpha}$, $\forall \alpha \in \Delta$, we obtain the following isomorphisms of abelian groups:
\begin{equation}
{\rm{H}}_{\omega_{0}} \cong {\rm{Pic}}^{0}_{\omega_{0}}(X_{P}) \cong H^{2}_{\omega_{0}}(X_{P},\mathbbm{Z})_{{\text{prim}}}
\end{equation}
The above description will be fundamental for the construction presented in the next section. For the sake of simplicity, fixed some $\gamma \in \Delta \backslash I$, we shall denote
\begin{equation}
\mathscr{O}_{\omega_{0}}(\gamma,\alpha):= \mathscr{O}_{\gamma}(-q_{\alpha}(\omega_{0})) \otimes \mathscr{O}_{\alpha}(q_{\gamma}(\omega_{0})), \ \ \forall \alpha \in \Delta \backslash I, \alpha \neq \gamma,
\end{equation}
thus ${\rm{Pic}}^{0}_{\omega_{0}}(X_{P}) = \big \langle \mathscr{O}_{\omega_{0}}(\gamma,\alpha) \ \big | \ \alpha \in \Delta \backslash I, \alpha \neq \gamma \big \rangle $. 
\end{remark}

\begin{remark}[Case $\mathbbm{F} = \mathbbm{Q}$] In the case that $[\omega_{0}] \in \mathcal{K}(X_{P})$ is a rational K\"{a}hler class, we can take some suitable positive integer $a \in \mathbbm{Z}$, such that $(X_{P},a\omega_{0})$ is a Hodge manifold. Then, applying on $(X_{P},a\omega_{0})$ the construction provided in the previous remark, we obtain a $\mathbbm{Z}$-basis for $H^{2}_{\omega_{0}}(X_{P},\mathbbm{Z})_{{\text{prim}}}$ given by
\begin{equation}
{\bf{\xi}}_{\alpha} := -q_{\alpha}(a\omega_{0})[{\bf{\Omega}}_{\gamma}] + q_{\gamma}(a\omega_{0})[{\bf{\Omega}}_{\alpha}], 
\end{equation}
for all $\forall \alpha \in \Delta \backslash I$, such that $\alpha \neq \gamma$.
\end{remark}

\begin{remark}
\label{relationcontraction}
Notice that, if $\omega_{0} \in \mathcal{K}(X_{P})$ is a K\"{a}hler class, such that $\omega_{0} = \lambda \omega_{1}$, where $\omega_{1}$ is an integral K\"{a}hler class and $\lambda > 0$ ($\lambda \in \mathbbm{R}$), then we have
\begin{equation}
\Lambda_{\omega_{0}}(-) = \frac{1}{\lambda}\Lambda_{\omega_{1}}(-),
\end{equation}
see for instance Eq. (\ref{contraction}). Thus, in this particular case, we have
\begin{equation}
\label{primitiveafterhomo}
H^{2}_{\omega_{0}}(X_{P},\mathbbm{Z})_{{\text{prim}}} = H^{2}_{\omega_{1}}(X_{P},\mathbbm{Z})_{{\text{prim}}} \ \ {\text{and}} \ \ {\rm{Pic}}^{0}_{\omega_{0}}(X_{P}) = {\rm{Pic}}^{0}_{\omega_{1}}(X_{P}),
\end{equation}
the equality of the right-hand side above is a consequence of Eq. (\ref{degreecontraction}).
\end{remark}

\section{Proof of main results}
Now we are in position to prove the main results.

\begin{theorem}
\label{Theorem1}
Let $X$ be a complex flag variety with Picard number $\varrho(X)>1$. Then, for every $k \in \mathbbm{Z}^{\times}$ and every real number $t < 1$, there exists $\lambda(k,t) \in \mathbbm{R}_{>0}$, and a short exact sequence of holomorphic vector bundles 
\begin{center}
\begin{tikzcd} 0 \arrow[r] & \mathscr{O}_{X}(k)  \arrow[r] & {\bf{E}}  \arrow[r]  & {\bf{F}} \arrow[r] & 0 \end{tikzcd}
\end{center}
satisfying the following properties:
\begin{enumerate}
\item[(1)] The holomorphic vector bundle ${\bf{F}} \to X$ admits a Hermitian structure ${\bf{h}}$ solving the Hermitian Yang-Mills equation
\begin{equation}
\sqrt{-1}\Lambda_{\omega_{0}}(F({\bf{h}})) = 0,
\end{equation}
where $[\omega_{0}] = \lambda(k,t)c_{1}(X)$. In particular, ${\bf{F}}$ is $[\omega_{0}]$-polystable;
\item[(2)] The manifold underlying the unitary frame bundle ${\rm{U}}({\bf{E}})$ is a compact non-K\"{a}hler Calabi-Yau manifold which admits a $t$-Gauduchon Ricci-flat Hermitian metric $\Omega$, such that the natural projection map
\begin{equation}
({\rm{U}}({\bf{E}}),\Omega) \to (X,\omega_{0}),
\end{equation}
is a Hermitian submersion;
\item[(3)] If $k = - {\bf{I}}(X)$, where ${\bf{I}}(X)$ is the Fano index of $X$, then the complex manifold underlying $\mathscr{O}_{X}(k)$ admits a complete Calabi-Yau metric.
\end{enumerate}
\end{theorem}

\begin{proof}
Considering $X = X_{P}$, let $[\omega_{0}] \in \mathcal{K}(X_{P})$ be a K\"{a}hler class, such that 
\begin{equation}
\Big [\frac{\omega_{0}}{2\pi} \Big] = \lambda c_{1}(X_{P}),
\end{equation}
for some $\lambda \in \mathbbm{R}_{>0}$. Let us denote by 
\begin{equation}
\mathscr{O}_{X_{P}}(1):= \frac{1}{{\bf{I}}(X_{P})}{\bf{K}}_{X_{P}}^{-1},
\end{equation}
the irreducible root of the anti-canonical bundle of $X_{P}$. For every $r \in \mathbbm{Z}_{>0}$, let 
\begin{equation}
{\bf{F}}_{1}, \dots, {\bf{F}}_{2r-1} \in {\rm{Pic}}_{\omega_{0}}^{0}(X_{P}), 
\end{equation}
be an arbitrary collection of non-trivial line bundles. For every $k \in \mathbbm{Z}$, $k\neq 0$, consider the holomorphic vector bundle ${\bf{E}} \to X_{P}$, defined by
\begin{equation}
{\bf{E}} := \mathscr{O}_{X_{P}}(k) \oplus \underbrace{{\bf{F}}_{1} \oplus \cdots \oplus {\bf{F}}_{2r-1}}_{{\bf{F}}},
\end{equation}
such that $\mathscr{O}_{X_{P}}(k):= \mathscr{O}_{X_{P}}(1)^{\otimes k}$. From this, we have a short exact sequence of holomorphic vector bundles
\begin{center}
\begin{tikzcd} 0 \arrow[r] & \mathscr{O}_{X_{P}}(k)  \arrow[r] & {\bf{E}}  \arrow[r]  & {\bf{F}} \arrow[r] & 0 \end{tikzcd}
\end{center}
By taking the Hermitian structure $\widehat{{\bf{h}}}$ on ${\bf{E}}$ defined by
\begin{equation}
\widehat{{\bf{h}}} := {\bf{h}}_{0} \oplus {\bf{h}}_{1} \oplus \cdots \oplus {\bf{h}}_{2r-1},
\end{equation}
such that ${\bf{h}}_{0}$ is a Hermitian structure on $\mathscr{O}_{X_{P}}(k)$ and ${\bf{h}}_{j}$ is a Hermitian structure on ${\bf{F}}_{j}$, $j = 1,\ldots,2r-1$, consider the principal torus bundle provided by the underlying unitary frame bundle ${\rm{U}}({\bf{E}})$ of ${\bf{E}}$, i.e.,
\begin{equation}
T^{2r} \hookrightarrow {\rm{U}}({\bf{E}}) \to X_{P}.
\end{equation}
In order to prove item (1) and (2), we can choose a Hermitian structure 
\begin{equation}
{\bf{h}}:= {\bf{h}}_{1} \oplus \cdots \oplus {\bf{h}}_{2r-1},
\end{equation}
on ${\bf{F}}$ and take principal connection
\begin{equation}
\Theta = \begin{pmatrix} \sqrt{-1}\Theta_{1} & \cdots & 0 \\
 \vdots & \ddots & \vdots \\
 0 & \cdots & \sqrt{-1}\Theta_{2r}\end{pmatrix} \in \Omega^{1} \big ({\rm{U}}({\bf{E}});{\text{Lie}}(T^{2r}) \big),
\end{equation}
satisfying the following:
\begin{enumerate}
\item ${\rm{d}}\Theta_{j} = \pi^{\ast}(\psi_{j})$, such that $\frac{\psi_{1}}{2\pi} \in c_{1}(\mathscr{O}_{X_{P}}(k)), \frac{\psi_{2}}{2\pi} \in c_{1}({\bf{F}}_{1}), \dots, \frac{\psi_{2r}}{2\pi} \in c_{1}({\bf{F}}_{2r-1})$;
\item $\psi_{1},\ldots,\psi_{2r}$ are $G$-invariant $(1,1)$-forms;
\item $\Lambda_{\omega_{0}}(\psi_{j}) = 0$, $\forall j = 2,\ldots, 2r$.
\item the curvature $F({\bf{h}})$ of the associated Chern connection $\nabla^{{\bf{h}}} \myeq {\rm{d}} + {\bf{h}}^{-1}\partial {\bf{h}}$ satisfies 
\begin{equation}
\sqrt{-1}F({\bf{h}}) = \begin{pmatrix} \psi_{2} & \cdots & 0 \\
 \vdots & \ddots & \vdots \\
 0 & \cdots & \psi_{2r}\end{pmatrix}.
\end{equation}
\end{enumerate}
 By construction, it follows that 
\begin{equation}
\sqrt{-1}\Lambda_{\omega_{0}}(F({\bf{h}})) = 0. 
\end{equation}
Thus, from Kobayashi-Hitchin correspondence \cite{donaldson1985anti,donaldson1987infinite},\cite{uhlenbeck1986existence}, we have that ${\bf{F}}$ is $[\omega_{0}]$-polystable. From above, we can equip ${\rm{U}}({\bf{E}})$ with a Hermitian structure $(\Omega,\mathbbm{J})$, such that 
\begin{enumerate}
\item[(a)] $\mathbbm{J}|_{\ker(\Theta)}$ is given by the lift of the complex structure of $X_{P}$;
\item[(b)] $\mathbbm{J}(\Theta_{2j-1}) = -\Theta_{2j}$, $j = 1,\ldots,r$;
\item[(c)] $\Omega := \pi^{\ast}(\omega_{0}) + \frac{1}{2}{\rm{tr}}\big ( \Theta \wedge \mathbbm{J}\Theta\big )$;
\end{enumerate}
If we consider the induced canonical connection $\nabla^{t}$ on $({\rm{U}}({\bf{E}}), \Omega,\mathbbm{J})$, for some $t \in \mathbbm{R}$, it follows that the associated Ricci form $\rho(\Omega,t)$ satisfies 
\begin{equation}
\rho(\Omega,t) = \pi^{\ast} \Big ( \rho_{\omega_0} + \frac{t-1}{2} \sum_{j = 1}^{2n} \Lambda_{\omega_{0}}(\psi_{j})\psi_{j}\Big ) = \pi^{\ast} \Big ( \rho_{\omega_0} + \frac{t-1}{2}  \Lambda_{\omega_{0}}(\psi_{1})\psi_{1}\Big ). 
\end{equation}
where $\rho_{\omega_0}$ is the Ricci form of $\omega_{0}$, see for instance Eq. (\ref{tgauduchonricvi}). If we chose $\omega_{0}$ as being the $G$-invariant representative of $2\pi \lambda c_{1}(X_{P})$, it follows from Eq. (\ref{ChernFlag}) that $\omega_{0} = \lambda\rho_{0}$ and $\rho_{\omega_0} = \rho_{0}$, thus
\begin{equation}
\rho(\Omega,t) = \pi^{\ast} \Big ( \rho_{0} + \frac{t-1}{2}  \Lambda_{\omega_{0}}(\psi_{1})\psi_{1}\Big ).
\end{equation}
On the other hand, since $\frac{\psi_{1}}{2\pi} \in c_{1}(\mathscr{O}_{X_{P}}(k))$ is $G$-invariant, it follows that 
\begin{equation}
\label{curvaturerelation}
\psi_{1} = \frac{k}{{\bf{I}}(X_{P})}\rho_{0} = \frac{ k}{\lambda {\bf{I}}(X_{P})}\omega_{0},
\end{equation}
thus, we obtain 
\begin{equation}
\Lambda_{\omega_{0}}(\psi_{1}) = \frac{k}{\lambda {\bf{I}}(X_{P})} \Lambda_{\omega_{0}}(\omega_{0}) = \frac{k \dim_{\mathbbm{C}}(X_{P})}{\lambda {\bf{I}}(X_{P})}.
\end{equation}
Hence, we conclude that 
\begin{equation}
\rho(\Omega,t) = 0 \iff \frac{1}{\lambda}\omega_{0} + \frac{t-1}{2} \bigg [ \frac{k^{2}\dim_{\mathbbm{C}}(X_{P})}{\lambda^{2}{\bf{I}}(X_{P})^{2}}\bigg]\omega_{0} = 0 \iff \lambda = \frac{1-t}{2} \bigg [ \frac{k^{2}\dim_{\mathbbm{C}}(X_{P})}{{\bf{I}}(X_{P})^{2}}\bigg].
\end{equation}
Therefore, for every $t < 1$ and $k \in \mathbbm{Z}$, $k\neq 0$, we can define 
\begin{equation}
\lambda(k,t) := \frac{1-t}{2} \bigg [ \frac{k^{2}\dim_{\mathbbm{C}}(X_{P})}{{\bf{I}}(X_{P})^{2}}\bigg],
\end{equation}
Considering the Hermitian metric $\Omega$ on $({\rm{U}}({\bf{E}}),\mathbbm{J})$, such that 
\begin{equation}
\Omega = \lambda(k,t)\pi^{\ast}(\rho_{0}) + \frac{1}{2}\big ( \Theta \wedge \mathbbm{J}\Theta\big ),
\end{equation}
by construction, we have $\rho(\Omega,t) = 0$. Further, from Eq. (\ref{tgauduchonricvi}) and Eq. (\ref{curvaturerelation}), it follows that
\begin{equation}
\rho(\Omega,1) = \pi^{\ast}(\rho_{0}) =  \frac{{\bf{I}}(X_{P})}{k} \pi^{\ast}(\psi_{1}) = \frac{{\bf{I}}(X_{P})}{k}{\rm{d}}\Theta_{1} \Rightarrow c_{1}({\rm{U}}({\bf{E}})) = \bigg [\frac{\rho(\Omega,1)}{2\pi}\bigg ] = 0,
\end{equation}
i.e., ${\rm{U}}({\bf{E}})$ is Calabi-Yau. Since ${\rm{U}}({\bf{E}})$ is a principal bundle with compact fiber over a compact base, we have that ${\rm{U}}({\bf{E}})$ is a compact manifold (e.g. \cite{naber1997topology}). Further, $({\rm{U}}({\bf{E}}),\mathbbm{J})$ cannot carry a K\"{a}hler structure for purely topological reasons (see for instance \cite[\S 1.7]{hofer1993remarks}, \cite{blanchard1952varietes}, \cite{poddar2018group}). From above we conclude the proof of item $(1)$ and item $(2)$. The proof of item $(3)$ follows from the following. If $k = - {\bf{I}}(X_{P})$, then $\mathscr{O}_{X_{P}}(k) = {\bf{K}}_{X_{P}}$. Thus, from the Calabi ansatz (e.g. \cite{calabi1979metriques}, \cite{CorreaGrama,Correa}), it follows that $\mathscr{O}_{X_{P}}(k)$ admits a complete Calabi-Yau metric. 
\end{proof}

\begin{remark}
\label{picardeequiv}
 In the setting of the above theorem, considering $\omega_{0} = \lambda(k,t)\rho_{0}$ and $\vartheta_{0}:=\frac{\rho_{0}}{2\pi}$, it follows from Remark \ref{primitiveafterhomo} that
\begin{equation}
 {\rm{Pic}}^{0}_{\omega_{0}}(X_{P}) = {\rm{Pic}}^{0}_{\vartheta_{0}}(X_{P}).
\end{equation}
Thus, once we know the generators of ${\rm{Pic}}^{0}_{\vartheta_{0}}(X_{P})$, we can describe explicitly the geometric structures provided by the last theorem. Notice that, since $[\vartheta_{0}]$ is an integral K\"{a}hler class, one can obtain the generators of ${\rm{Pic}}^{0}_{\vartheta_{0}}(X_{P})$ through a procedure as in Remark \ref{integralcase}.
\end{remark}

From Theorem \ref{Theorem1}, we have the following corollary.
\begin{corollary}
In the setting of the previous theorem, we have the following:
\begin{enumerate}
\item If $t = -1$, then $({\rm{U}}({\bf{E}}),\Omega)$ is Strominger-Bismut Ricci-flat; 
\item If $t = 0$, then $({\rm{U}}({\bf{E}}),\Omega)$ is Lichnerowicz Ricci-flat.
\end{enumerate}
\end{corollary}

\begin{theorem}
\label{theorem2proof}
Let $X$ be a complex flag variety with Picard number $\varrho(X)>1$. Then, for every $r \in \mathbbm{Z}_{>0}$ and every $G$-invariant K\"{a}hler metric $\omega_{0}$, there exists a holomorphic vector bundle ${\bf{F}} \to X$, such that $\rank({\bf{F}}) = 2r$, satisfying the following properties:
\begin{enumerate}
\item ${\bf{F}} \to X$ admits a Hermitian structure ${\bf{h}}$ solving the Hermitian Yang-Mills equation
\begin{equation}
\sqrt{-1}\Lambda_{\omega_{0}}(F({\bf{h}})) = 0.
\end{equation}
 In particular, ${\bf{F}}$ is $[\omega_{0}]$-polystable;
\item The manifold underlying the unitary frame bundle ${\rm{U}}({\bf{F}})$ is a compact complex non-K\"{a}hler manifold which admits a balanced Hermitian metric $\Omega$, such that the natural projection map
\begin{equation}
({\rm{U}}({\bf{F}}),\Omega) \to (X,\omega_{0}),
\end{equation}
is a Hermitian submersion.
\end{enumerate}
\end{theorem}
\begin{proof}
Given $r \in \mathbbm{Z}_{>0}$, and fixed some $G$-invariant K\"{a}hler metric $\omega_{0}$ on $X_{P}$, we define
\begin{equation}
{\bf{F}} := {\bf{F}}_{1} \oplus \cdots \oplus {\bf{F}}_{2r}, 
\end{equation}
such that ${\bf{F}}_{1}, \ldots, {\bf{F}}_{2r} \in {\rm{Pic}}_{\omega_{0}}^{0}(X_{P})$. In order to prove item (1) and (2), we can take a Hermitian structure ${\bf{h}}$ on ${\bf{F}}$ and a principal connection
\begin{equation}
\Theta = \begin{pmatrix} \sqrt{-1}\Theta_{1} & \cdots & 0 \\
 \vdots & \ddots & \vdots \\
 0 & \cdots & \sqrt{-1}\Theta_{2r}\end{pmatrix} \in \Omega^{1} \big({\rm{U}}({\bf{F}});{\text{Lie}}(T^{2r}) \big),
\end{equation}
on ${\rm{U}}({\bf{F}}) = {\rm{U}}({\bf{F}},{\bf{h}})$ satisfying the following:
\begin{enumerate}
\item The curvature $F({\bf{h}})$ of the associated Chern connection $\nabla^{{\bf{h}}} \myeq {\rm{d}} + {\bf{h}}^{-1}\partial {\bf{h}}$ satisfies 
\begin{equation}
\sqrt{-1}F({\bf{h}}) = \begin{pmatrix} \psi_{1} & \cdots & 0 \\
 \vdots & \ddots & \vdots \\
 0 & \cdots & \psi_{2r}\end{pmatrix};
\end{equation}
\item ${\rm{d}}\Theta_{j} = \pi^{\ast}(\psi_{j})$, such that $\frac{\psi_{1}}{2\pi} \in c_{1}({\bf{F}}_{1}), \dots, \frac{\psi_{2r}}{2\pi} \in c_{1}({\bf{F}}_{2r})$;
\item $\psi_{1},\ldots,\psi_{2r}$ are $G$-invariant $(1,1)$-forms;
\item $\Lambda_{\omega_{0}}(\psi_{j}) = 0$, $\forall j = 1,\ldots, 2r$.
\end{enumerate}
By construction, it follows that 
\begin{equation}
\sqrt{-1}\Lambda_{\omega_{0}}(F({\bf{h}})) = 0. 
\end{equation}
In particular, we have that ${\bf{F}}$ is $[\omega_{0}]$-polystable. Further, we can equip ${\rm{U}}({\bf{F}})$ with a Hermitian structure $(\Omega,\mathbbm{J})$, such that 
\begin{enumerate}
\item[(a)] $\mathbbm{J}|_{\ker(\Theta)}$ is given by the lift of the complex structure of $X_{P}$;
\item[(b)] $\mathbbm{J}(\Theta_{2j-1}) = -\Theta_{2j}$, $j = 1,\ldots,r$;
\end{enumerate}
and
\begin{equation}
 \Omega = \pi^{\ast}(\omega_{0}) + \frac{1}{2}{\rm{tr}} \big ( \Theta \wedge \mathbbm{J}\Theta\big ).
\end{equation}
By construction, it follows from Eq. (\ref{coderivativekahlerform}) that
\begin{equation}
\delta \Omega = \pi^{\ast} (\delta \omega_{0}) + \sum_{j = 1}^{2r} \pi^{\ast} \big (\Lambda_{\omega_{0}}(\psi_{j}) \big ) \Theta_{j} = 0.
\end{equation}
Hence, $({\rm{U}}({\bf{F}}), \Omega,\mathbbm{J})$ is a balanced Hermitian manifold. From a similar argument as in the proof of item (2) of the previous theorem, we have that ${\rm{U}}({\bf{F}})$ is compact and cannot carry a K\"{a}hler structure for purely topological reasons.
\end{proof}

\subsection{General construction} In this subsection, we will provide a detailed description of all Hermitian metrics obtained from our main results. In the setting of Theorem \ref{Theorem1}, let $X_{P}$ be a complex flag variety and let us fix the K\"{a}hler metric $\omega_{0} = \lambda(k,t)\rho_{0}$, such that 
\begin{equation} 
\label{einsteinconstant}
\Big [\frac{\rho_{0}}{2 \pi} \Big ] = c_{1}(X_{P}) \ \ {\text{and}}  \ \ \lambda(k,t) = \frac{1-t}{2} \bigg [ \frac{k^{2}\dim_{\mathbbm{C}}(X_{P})}{{\bf{I}}(X_{P})^{2}}\bigg],
\end{equation}
for some $k \in \mathbbm{Z}$, and some real number $t< 1$. From Remark \ref{picardeequiv}, we have 
\begin{equation}
 {\rm{Pic}}^{0}_{\omega_{0}}(X_{P}) = {\rm{Pic}}^{0}_{\vartheta_{0}}(X_{P}),
\end{equation}
where $\vartheta_{0}:=\frac{\rho_{0}}{2\pi}$. Hence, once we describe the Hermitian Yang-Mills instantons for every ${\bf{F}} \in {\rm{Pic}}^{0}_{\vartheta_{0}}(X_{P})$, we are able to describe precisely the Hermitian metrics provided by Theorem \ref{Theorem1}. From Remark \ref{integralcase}, fixed some $\gamma \in \Delta \backslash I$, we have 
\begin{center}
${\rm{Pic}}^{0}_{\vartheta_{0}}(X_{P}) = \big \langle \mathscr{O}_{\vartheta_{0}}(\gamma,\alpha) \ \big | \ \alpha \in \Delta \backslash I, \alpha \neq \gamma \big \rangle,$
\end{center}
where $\mathscr{O}_{\vartheta_{0}}(\gamma,\alpha):= \mathscr{O}_{\gamma}(-q_{\alpha}(\vartheta_{0})) \otimes \mathscr{O}_{\alpha}(q_{\gamma}(\vartheta_{0})), \ \ \forall \alpha \in \Delta \backslash I, \alpha \neq \gamma$, such that 
\begin{equation}
q_{\alpha}(\vartheta_{0}) := \frac{\mathcal{Q}_{\vartheta_{0}}({\bf{\Omega}}_{\alpha},\vartheta_{0})}{\tau([\vartheta_{0}])}, \forall \alpha \in \Delta \backslash I,
\end{equation}
such that 
\begin{equation}
\tau([\vartheta_{0}]) := {\rm{G.C.D.}} \Big ( \mathcal{Q}_{\vartheta_{0}}({\bf{\Omega}}_{\alpha},\vartheta_{0}) \ \bigg | \ \alpha \in \Delta \backslash I \Big ).
\end{equation}
Notice that $q_{\alpha}(\omega_{0})$, $\alpha \in \Delta \backslash I$, might not be an integer number, that is the reason why we replace $\omega_{0}$ by $\vartheta_{0}$. From this, given ${\bf{F}} \in {\rm{Pic}}^{0}_{\vartheta_{0}}(X_{P}) $, we have 
\begin{equation}
{\bf{F}} = \bigotimes_{\alpha \in \Delta \backslash I, \alpha \neq \gamma} \mathscr{O}_{\vartheta_{0}}(\gamma,\alpha)^{\otimes k_{\alpha}},
\end{equation}
such that $k_{\alpha} = \langle \lambda({\bf{F}}), \alpha^{\vee}\rangle \in \mathbbm{Z}$, $\forall \alpha \in \Delta \backslash I, \alpha \neq \gamma$. Let $U \subset X_{P}$ be an open set which trivializes both $P \hookrightarrow G^{\mathbbm{C}} \to X_{P}$ and ${\bf{F}} \to X_{P}$. By taking a local section ${\bf{s}}_{U} \colon U \to G^{\mathbbm{C}}$, and considering fiber coordinates $w$ on ${\bf{F}}|_{U}$, we can define a Hermitian structure ${\bf{h}}_{U}$ on ${\bf{F}}|_{U}$ by setting
\begin{equation}
{\bf{h}}_{U} := \frac{w\overline{w}}{\displaystyle \prod_{\alpha \in \Delta \backslash I, \alpha \neq \gamma} \big | \big | {\bf{s}}_{U}v_{\varpi_{\gamma}}^{+}\big | \big |^{-2k_{\alpha}q_{\alpha}(\vartheta_{0})}  \big | \big | {\bf{s}}_{U}v_{\varpi_{\alpha}}^{+}\big | \big |^{2k_{\alpha}q_{\gamma}(\vartheta_{0})}},
\end{equation}
here we consider on each irreducible $\mathfrak{g}^{\mathbbm{C}}$-module $V(\varpi_{\alpha})$, $\alpha \in \Delta \backslash I$, the norm $||\cdot ||$ induced by some $G$-invariant inner product. By patching together the local Hermitian structures described as above, we obtain a Hermitian structure ${\bf{h}}$ on ${\bf{F}}$, such that ${\bf{h}}|_{U} = {\bf{h}}_{U}$, for every trivializing open set $U \subset X_{P}$. In this case, a straightforward computation shows that the associated Chern connection $\nabla^{{\bf{h}}}$ is locally given by $\nabla^{{\bf{h}}}|_{U} = {\rm{d}} + A_{U}$, such that 
\begin{equation}
A_{U} = \sum_{\alpha \in \Delta \backslash I, \alpha \neq \gamma} \Big ( k_{\alpha}q_{\alpha}(\vartheta_{0}) \partial \log \big ( ||{\bf{s}}_{U}v_{\varpi_{\gamma}}^{+}||^{2}\big ) - k_{\alpha}q_{\gamma}(\vartheta_{0}) \partial \log \big ( ||{\bf{s}}_{U}v_{\varpi_{\alpha}}^{+}||^{2}\big )\Big ).
\end{equation}
From above, we conclude that the curvature $F({\bf{h}})$ of $\nabla^{{\bf{h}}}$ satisfies
\begin{equation}
\frac{\sqrt{-1}}{2\pi}F({\bf{h}}) = \sum_{\alpha \in \Delta \backslash I, \alpha \neq \gamma} k_{\alpha} \Big (q_{\gamma}(\vartheta_{0}){\bf{\Omega}}_{\alpha} - q_{\alpha}(\vartheta_{0}){\bf{\Omega}}_{\gamma}\Big ).
\end{equation}
Since 
\begin{equation}
\Lambda_{\vartheta_{0}} \Big ( q_{\gamma}(\vartheta_{0}){\bf{\Omega}}_{\alpha} - q_{\alpha}(\vartheta_{0}){\bf{\Omega}}_{\gamma}\Big) = \frac{\Big ( q_{\gamma}(\vartheta_{0})\mathcal{Q}_{\vartheta_{0}}({\bf{\Omega}}_{\alpha},\vartheta_{0}) - q_{\alpha}(\vartheta_{0})\mathcal{Q}_{\vartheta_{0}}({\bf{\Omega}}_{\gamma},\vartheta_{0})\Big)}{(n-1)! {\rm{Vol}}(X_{P},\vartheta_{0})} = 0,
\end{equation}
for every $\alpha \in \Delta \backslash I, \alpha \neq \gamma$, see Eq. (\ref{contractionvolume}), it follows from Remark \ref{relationcontraction} that
\begin{equation}
\sqrt{-1}\Lambda_{\omega_{0}}(F({\bf{h}})) = 0.
\end{equation}
Considering now the holomorphic vector bundle ${\bf{E}} \to X_{P}$, defined by
\begin{equation}
{\bf{E}} := \mathscr{O}_{X_{P}}(k) \oplus \underbrace{{\bf{F}}_{1} \oplus \cdots \oplus {\bf{F}}_{2r-1}}_{{\bf{F}}},
\end{equation}
such that $\mathscr{O}_{X_{P}}(k):= \mathscr{O}_{X_{P}}(1)^{\otimes k}$ and ${\bf{F}}_{1}, \dots, {\bf{F}}_{2r-1} \in {\rm{Pic}}_{\omega_{0}}^{0}(X_{P})$, we can construct a principal connection $\Theta \in \Omega^{1} ({\rm{U}}({\bf{E}});{\text{Lie}}(T^{2r}))$ satisfying the properties required in the proof of Theorem \ref{Theorem1} in the following way. By taking an open set $U \subset X_{P}$ which trivializes both $P \hookrightarrow G^{\mathbbm{C}} \to X_{P}$ and $ T^{2r} \hookrightarrow {\rm{U}}({\bf{E}}) \to X_{P}$, we define
\begin{equation}
\Theta_{U} := \pi^{\ast} \mathcal{A}_{U} + a_{U}^{-1}{\rm{d}}a_{U},
\end{equation}
where
\begin{equation}
\mathcal{A}_{U} = \begin{pmatrix} \mathcal{A}_{1} & \cdots & 0 \\
 \vdots & \ddots & \vdots \\
 0 & \cdots & \mathcal{A}_{2r}\end{pmatrix} \ \ \ {\text{and}} \ \ \ a_{U} = \begin{pmatrix} a_{1} & \cdots & 0 \\
 \vdots & \ddots & \vdots \\
 0 & \cdots & a_{2r}\end{pmatrix}
\end{equation}
such that 
\begin{itemize}
\item $a_{1},\ldots,a_{2r}$ are fiber coordinates,
\item $\displaystyle \mathcal{A}_{1} = \sum_{\alpha \in \Delta \backslash I}\frac{k \langle \delta_{P},\alpha^{\vee} \rangle}{I(X_{P})}  \frac{1}{2} {\rm{d}}^{c}\log \big ( ||{\bf{s}}_{U}v_{\varpi_{\alpha}}^{+}||^{2}\big )$,
\item $\displaystyle \mathcal{A}_{j+1} = \sum_{\alpha \in \Delta \backslash I, \alpha \neq \gamma}  \langle \lambda({\bf{F}}_{j}),\alpha^{\vee} \rangle \frac{1}{2} {\rm{d}}^{c}\log \Bigg ( \displaystyle \frac{||{\bf{s}}_{U}v_{\varpi_{\alpha}}^{+}||^{2q_{\gamma}(\vartheta_{0})}}{||{\bf{s}}_{U}v_{\varpi_{\gamma}}^{+}||^{2q_{\alpha}(\vartheta_{0})}}\Bigg)$, $j=1,\ldots,2r-1$,
\end{itemize}
where ${\bf{s}}_{U} \colon U \to G^{\mathbbm{C}}$ is some local section, and ${\rm{d}}^{c} = \sqrt{-1}(\overline{\partial} - \partial)$ . If $V \subset X_{P}$ is another trivializing open set as above, such that $U \cap V \neq \emptyset$, considering the associated local data $\mathcal{A}_{V} = {\text{diag}} \{ \mathcal{A}'_{1},\ldots,\mathcal{A}'_{2r} \}$ and $a_{V} = {\text{diag}} \{a'_{1},\ldots,a'_{2r}\}$, we have the following relations on the overlap $U \cap V$:
\begin{enumerate}
\item $\pi^{\ast}\mathcal{A}_{U} = \pi^{\ast}\mathcal{A}_{V} - t_{UV}^{-1}{\rm{d}}t_{UV}$, such that $t_{UV} =  {\text{diag}} \{ t_{UV}^{(1)},\ldots, t_{UV}^{(2r)}\}$;
\item $a_{U} = a_{V}t_{UV}$, thus $a_{U}^{-1}{\rm{d}}a_{U} = a_{U}^{-1}{\rm{d}}a_{U} + t_{UV}^{-1}{\rm{d}}t_{UV}$.
\end{enumerate}
It is worth pointing out that 
\begin{center}
$\displaystyle t_{UV}^{(j)}  = \frac{g_{UV}^{(j)}}{|g_{UV}^{(j)}|}$, 
\end{center}
such that $g_{UV}^{(1)}$ is a transition function of $\mathscr{O}_{X_{P}}(k)$, and $g_{UV}^{(j)}$ is a transition function of ${\bf{F}}_{j-1}$, for $j = 2,\ldots,2r$. By taking a suitable open cover $\mathscr{U}$ of $X_{P}$, we obtain a system of (local) gauge fields $\{\Theta_{U}\}_{U \in \mathscr{U}}$ which gives rise to a principal connection 
\begin{equation}
\Theta = \begin{pmatrix} \sqrt{-1}\Theta_{1} & \cdots & 0 \\
 \vdots & \ddots & \vdots \\
 0 & \cdots & \sqrt{-1}\Theta_{2r}\end{pmatrix} \in \Omega^{1} \big ({\rm{U}}({\bf{E}});{\text{Lie}}(T^{2r}) \big),
\end{equation}
satisfying the following:
\begin{enumerate}
\item ${\rm{d}}\Theta_{j} = \pi^{\ast}(\psi_{j})$, such that $\frac{\psi_{1}}{2\pi} \in c_{1}(\mathscr{O}_{X_{P}}(k)), \frac{\psi_{2}}{2\pi} \in c_{1}({\bf{F}}_{1}), \dots, \frac{\psi_{2r}}{2\pi} \in c_{1}({\bf{F}}_{2r-1})$;
\item $\psi_{1},\ldots,\psi_{2r}$ are $G$-invariant $(1,1)$-forms;
\item $\Lambda_{\omega_{0}}(\psi_{j}) = 0$, $\forall j = 2,\ldots, 2r$.
\end{enumerate}
Combining the above data with $\omega_{0} = \lambda(k,t)\rho_{0}$ (see Eq. (\ref{einsteinconstant})), we obtain a $t$-Gauduchon Ricci-flat Hermitian structure $(\Omega,\mathbbm{J})$ on ${\rm{U}}({\bf{E}})$, such that 
\begin{enumerate}
\item[(a)] $\mathbbm{J}|_{\ker(\Theta)}$ is given by the lift of the complex structure of $X_{P}$;
\item[(b)] $\mathbbm{J}(\Theta_{2j-1}) = -\Theta_{2j}$, $j = 1,\ldots,r$;
\end{enumerate}
and
\begin{equation}
 \Omega = \pi^{\ast}(\omega_{0}) + \frac{1}{2}{\rm{tr}} \big ( \Theta \wedge \mathbbm{J}\Theta\big ).
\end{equation}
By means of a similar argument one can also describe explicitly the balanced Hermitian structure $(\Omega,\mathbbm{J})$ provided by Theorem \ref{theorem2proof}. 

\section{Examples on $T^{2}$-bundles over ${\mathbbm{P}}(T_{{\mathbbm{P}^{2}}})$} In this subsection, we will explore some explicit computations in concrete cases. The main purpose is to illustrate our main results providing some new examples of $t$-Gauduchon Ricci-flat Hermitian structures and balanced Hermitian structures on principal $T^{2}$-bundles over the Fano threefold ${\mathbbm{P}}(T_{{\mathbbm{P}^{2}}})$. Consider the complex Lie group $G^{\mathbbm{C}} = {\rm{SL}}_{3}(\mathbbm{C})$. In this case, the structure of the associated Lie algebra $\mathfrak{sl}_{3}(\mathbbm{C})$ can be completely determined by means of its Dynkin diagram
\begin{center}
${\dynkin[labels={\alpha_{1},\alpha_{2}},scale=3]A{oo}} $
\end{center}
Fixed the Cartan subalgebra $\mathfrak{h} \subset \mathfrak{sl}_{3}(\mathbbm{C})$ of diagonal matrices, we have the associated simple root system given by $\Delta  = \{\alpha_{1},\alpha_{2}\}$, such that 
\begin{center}
$\alpha_{j}({\rm{diag}}(d_{1},d_{2},d_{3})) = d_{j} - d_{j+1}$, $j = 1,2$,
\end{center}
$\forall {\rm{diag}}(d_{1},d_{2},d_{3}) \in \mathfrak{h}$. The set of positive roots in this case is given by 
\begin{center}
$\Phi^{+} = \{\alpha_{1}, \alpha_{2}, \alpha_{3} = \alpha_{1} + \alpha_{2}\}$. 
\end{center}
Considering the Cartan-Killing form\footnote{In this case, we have $\kappa(X,Y) = 6{\rm{tr}}(XY), \forall X,Y \in \mathfrak{sl}_{3}(\mathbbm{C})$, see for instance \cite[Chapter 10, \S 4]{procesi2007lie}.} $\kappa(X,Y) = {\rm{tr}}({\rm{ad}}(X){\rm{ad}}(Y)), \forall X,Y \in \mathfrak{sl}_{3}(\mathbbm{C})$, it follows that $\alpha_{j} = \kappa(\cdot,h_{\alpha_{j}})$, $j =1,2,3$, such that\footnote{Notice that $\langle \alpha_{j},\alpha_{j} \rangle = \alpha_{j}(h_{\alpha_{j}}) = \frac{1}{3}, \forall j = 1,2,3.$} 
\begin{equation}
h_{\alpha_{1}} =\frac{1}{6}(E_{11} - E_{22}), \ \ h_{\alpha_{2}} =\frac{1}{6}(E_{22} - E_{33}), \ \ h_{\alpha_{3}} =\frac{1}{6}(E_{11} - E_{33}),
\end{equation}
here we consider the matrices $E_{ij}$ as being the elements of the standard basis of ${\mathfrak{gl}}_{3}(\mathbbm{C})$. Moreover, we have the following relation between simple roots and fundamental weights:
\begin{center}
$\displaystyle{\begin{pmatrix}
\alpha_{1} \\ 
\alpha_{2}
\end{pmatrix} = \begin{pmatrix} \ \ 2 & -1 \\
-1 & \ \ 2\end{pmatrix} \begin{pmatrix}
\varpi_{\alpha_{1}} \\ 
\varpi_{\alpha_{2}}
\end{pmatrix}, \ \ \ \begin{pmatrix}
\varpi_{\alpha_{1}} \\ 
\varpi_{\alpha_{2}}
\end{pmatrix} = \frac{1}{3}\begin{pmatrix} 2 & 1 \\
1 & 2\end{pmatrix} \begin{pmatrix}
\alpha_{1} \\ 
\alpha_{2}
\end{pmatrix}},$
\end{center}
here we consider the Cartan matrix $C = (C_{ij})$ of $\mathfrak{sl}_{3}(\mathbbm{C})$ given by 
\begin{equation}
\label{Cartanmatrix}
C = \begin{pmatrix}
 \ \ 2 & -1 \\
-1 & \ \ 2 
\end{pmatrix}, \ \ C_{ij} = \frac{2\langle \alpha_{i}, \alpha_{j} \rangle}{\langle \alpha_{j}, \alpha_{j} \rangle},
\end{equation}
for more details on the above subject, see for instance \cite{Humphreys}. Fixed the standard Borel subgroup $B \subset {\rm{SL}}_{3}(\mathbbm{C})$, i.e.,
\begin{center}
$B = \Bigg \{ \begin{pmatrix} \ast & \ast & \ast \\
0 & \ast & \ast \\
0 & 0 & \ast \end{pmatrix} \in {\rm{SL}}_{3}(\mathbbm{C})\Bigg\},$
\end{center}
we consider the flag variety obtained from $I = \emptyset$, i.e., the homogeneous Fano threefold given by the Wallach space ${\mathbbm{P}}(T_{{\mathbbm{P}^{2}}}) = {\rm{SL}}_{3}(\mathbbm{C})/B$. In this particular case, we have the following facts:
\begin{enumerate}
\item[(i)] $H^{2}({\mathbbm{P}}(T_{{\mathbbm{P}^{2}}}),\mathbbm{R}) = H^{1,1}({\mathbbm{P}}(T_{{\mathbbm{P}^{2}}}),\mathbbm{R}) = \mathbbm{R}[{\bf{\Omega}}_{\alpha_{1}}] \oplus \mathbbm{R}[{\bf{\Omega}}_{\alpha_{2}}]$;
\item[(ii)] ${\rm{Pic}}({\mathbbm{P}}(T_{{\mathbbm{P}^{2}}})) = \Big \{ \mathscr{O}_{\alpha_{1}}(s_{1}) \otimes \mathscr{O}_{\alpha_{1}}(s_{1}) \ \Big | \ s_{1}, s_{2} \in \mathbbm{Z}\Big \}$.
\end{enumerate}
Let $\vartheta_{0}$ be the unique ${\rm{SU}}(3)$-invariant K\"{a}hler metric on ${\mathbbm{P}}(T_{{\mathbbm{P}^{2}}})$, such that $[\vartheta_{0}] = c_{1}({\mathbbm{P}}(T_{{\mathbbm{P}^{2}}}))$. Since $\lambda({\bf{K}}_{{\mathbbm{P}}(T_{{\mathbbm{P}^{2}}})}^{-1}) = \delta_{B} = 2(\varpi_{\alpha_{1}} + \varpi_{\alpha_{2}})$, from Eq. (\ref{ChernFlag}), it follows that 
\begin{equation}
\vartheta_{0} = \langle \delta_{B}, \alpha_{1}^{\vee} \rangle {\bf{\Omega}}_{\alpha_{1}} + \langle \delta_{B}, \alpha_{2}^{\vee} \rangle {\bf{\Omega}}_{\alpha_{2}} = 2 \big ({\bf{\Omega}}_{\alpha_{1}} + {\bf{\Omega}}_{\alpha_{2}}\big),
\end{equation}
in particular, notice that $\lambda([\vartheta_{0}]) = \delta_{B}$. In order to construct the Hermitian metrics provided by Theorem \ref{Theorem1} and Theorem \ref{theorem2proof}, we need to describe the subgroup
\begin{equation}
{\rm{Pic}}^{0}_{\vartheta_{0}}({\mathbbm{P}}(T_{{\mathbbm{P}^{2}}})) = \Big \{ {\bf{L}} \in {\rm{Pic}}({\mathbbm{P}}(T_{{\mathbbm{P}^{2}}})) \ \Big | \ \deg_{\vartheta_{0}}({\bf{L}}) = 0\Big \}.
\end{equation}
Given ${\bf{L}} = \mathscr{O}_{\alpha_{1}}(s_{1}) \otimes \mathscr{O}_{\alpha_{1}}(s_{1})$, we have 
\begin{equation}
\deg_{\vartheta_{0}}({\bf{L}}) = 0 \iff s_{1}\mathcal{Q}_{\vartheta_{0}}({\bf{\Omega}}_{\alpha_{1}},\vartheta_{0}) + s_{2}\mathcal{Q}_{\vartheta_{0}}({\bf{\Omega}}_{\alpha_{2}},\vartheta_{0}) = 0.
\end{equation}
From Eq. (\ref{contractionvolume}), it follows that 
\begin{equation}
\deg_{\vartheta_{0}}({\bf{L}}) = 0 \iff s_{1}\Lambda_{\vartheta_{0}}({\bf{\Omega}}_{\alpha_{1}}) + s_{2}\Lambda_{\vartheta_{0}}({\bf{\Omega}}_{\alpha_{2}}) = 0.
\end{equation}
The coefficients of the above linear equation can be computed by means of Eq. (\ref{contraction}), and using the that
\begin{equation}
\langle \lambda,\alpha_{1}^{\vee} \rangle = a, \ \ \langle \lambda,\alpha_{2}^{\vee} \rangle = b, \ \ \langle \lambda,\alpha_{3}^{\vee} \rangle = a + b,
\end{equation}
for every $\lambda = a\varpi_{\alpha_{1}} + b \varpi_{\alpha_{2}}$, $a,b \in \mathbbm{Z}$. Therefore, we have 
\begin{equation}
\deg_{\vartheta_{0}}({\bf{L}}) = 0 \iff \frac{3}{4}(s_{1} + s_{2}) = 0.
\end{equation}
From above, we conclude that ${\rm{Pic}}^{0}_{\vartheta_{0}}({\mathbbm{P}}(T_{{\mathbbm{P}^{2}}})) = \big \langle \mathscr{O}_{\alpha_{1}}(-1)\otimes \mathscr{O}_{\alpha_{2}}(1) \big \rangle $. Given $k \in \mathbbm{Z}^{\times}$ and a real number $t < 1$, consider 
\begin{equation}
\lambda(k,t) := \frac{1-t}{2} \bigg [ \frac{k^{2}\dim_{\mathbbm{C}}({\mathbbm{P}}(T_{{\mathbbm{P}^{2}}}))}{{\bf{I}}({\mathbbm{P}}(T_{{\mathbbm{P}^{2}}}))^{2}}\bigg] = \frac{3}{8}(1-t)k^{2},
\end{equation}
and define $\omega_{0} := \lambda(k,t)\rho_{0}$, such that $\rho_{0} = 2 \pi \vartheta_{0}$ (Eq. (\ref{riccinorm})). Given ${\bf{F}} \in {\rm{Pic}}^{0}_{\vartheta_{0}}({\mathbbm{P}}(T_{{\mathbbm{P}^{2}}}))$, such that 
\begin{equation}
{\bf{F}} = \mathscr{O}_{\alpha_{1}}(-\ell)\otimes \mathscr{O}_{\alpha_{2}}(\ell),
\end{equation}
for some $\ell \in \mathbbm{Z}$, $\ell \neq 0$, we construct the short exact sequence of holomorphic vector bundles
\begin{center}
\begin{tikzcd} 0 \arrow[r] & \mathscr{O}_{{\mathbbm{P}}(T_{{\mathbbm{P}^{2}}})}(k)  \arrow[r] & {\bf{E}}  \arrow[r]  & {\bf{F}} \arrow[r] & 0 \end{tikzcd}
\end{center}
In other words, we consider 
\begin{equation}
{\bf{E}} =  \mathscr{O}_{{\mathbbm{P}}(T_{{\mathbbm{P}^{2}}})}(k) \oplus \underbrace{\Big (  \mathscr{O}_{\alpha_{1}}(-\ell)\otimes \mathscr{O}_{\alpha_{2}}(\ell)\Big )}_{{\bf{F}}}.
\end{equation}
Notice that, since ${\bf{I}}({\mathbbm{P}}(T_{{\mathbbm{P}^{2}}})) = 2$ (Remark \ref{Fanoindex}), it follows that 
\begin{equation}
\mathscr{O}_{{\mathbbm{P}}(T_{{\mathbbm{P}^{2}}})}(1):= \frac{1}{{\bf{I}}({\mathbbm{P}}(T_{{\mathbbm{P}^{2}}}))}{\bf{K}}_{{\mathbbm{P}}(T_{{\mathbbm{P}^{2}}})}^{-1} =  \mathscr{O}_{\alpha_{1}}(1)\otimes \mathscr{O}_{\alpha_{2}}(1)
\end{equation}
From above, given an open set $U \subset {\mathbbm{P}}(T_{{\mathbbm{P}^{2}}})$ which trivializes both ${\bf{E}} \to {\mathbbm{P}}(T_{{\mathbbm{P}^{2}}})$ and $B \hookrightarrow {\rm{SL}}_{3}(\mathbbm{C}) \to {\mathbbm{P}}(T_{{\mathbbm{P}^{2}}})$, and denoting by $(u,w)$ the fiber coordinate in ${\bf{E}}|_{U}$, we can construct a Hermitian structure $\bf{h}$ on ${\bf{E}}$ by gluing the local Hermitian structures 
\begin{equation}
\label{locHermitian}
{\bf{h}}_{U} = \begin{pmatrix} v & w\end{pmatrix} \begin{pmatrix} \displaystyle \frac{1}{||{\bf{s}}_{U}v_{\varpi_{\alpha_{1}}}^{+}||^{2k} ||{\bf{s}}_{U}v_{\varpi_{\alpha_{2}}}^{+}||^{2k}} & 0 \\ 
0 & \displaystyle \frac{ ||{\bf{s}}_{U}v_{\varpi_{\alpha_{1}}}^{+}||^{2\ell}} {||{\bf{s}}_{U}v_{\varpi_{\alpha_{2}}}^{+}||^{2\ell}} \end{pmatrix} \begin{pmatrix} \overline{v} \\ \overline{w}\end{pmatrix},
\end{equation}
where ${\bf{s}}_{U} \colon U \subset {\mathbbm{P}}(T_{{\mathbbm{P}^{2}}})\to {\rm{SL}}_{3}(\mathbbm{C})$ is some local section, here we consider $||\cdot||$ defined by some fixed ${\rm{SU}}(3)$-invariant inner product on $V(\varpi_{\alpha_{k}})$, $k = 1,2$. Notice that ${\bf{h}} = {\bf{h}}_{1}\oplus {\bf{h}}_{2}$, where ${\bf{h}}_{1}$ and ${\bf{h}}_{2}$ are Hermitian structures on $\mathscr{O}_{{\mathbbm{P}}(T_{{\mathbbm{P}^{2}}})}(k)$ and ${\bf{F}}$, respectively, such that
\begin{equation}
{\bf{h}}_{1}|_{U} = \frac{v\overline{v}}{||{\bf{s}}_{U}v_{\varpi_{\alpha_{1}}}^{+}||^{2k} ||{\bf{s}}_{U}v_{\varpi_{\alpha_{2}}}^{+}||^{2k}} \ \ \ {\text{and}} \ \ \ {\bf{h}}_{2}|_{U} = \frac{ ||{\bf{s}}_{U}v_{\varpi_{\alpha_{1}}}^{+}||^{2\ell}} {||{\bf{s}}_{U}v_{\varpi_{\alpha_{2}}}^{+}||^{2\ell}} w\overline{w}.
\end{equation}
\begin{remark}
\label{HYMinstantonlinebundle}
Consider $U = U^{-}(B) \subset {\mathbbm{P}}(T_{{\mathbbm{P}^{2}}})$, such that
\begin{equation}
U^{-}(B) = \Bigg \{ \begin{pmatrix}
1 & 0 & 0 \\
z_{1} & 1 & 0 \\                  
z_{2}  & z_{3} & 1
 \end{pmatrix}B \ \Bigg | \ z_{1},z_{2},z_{3} \in \mathbbm{C} \Bigg \} \ \ \ ({\text{opposite big cell}}),
\end{equation}
The open set above is dense and contractible, so it trivializes the desired bundles over ${\mathbbm{P}}(T_{{\mathbbm{P}^{2}}})$. By taking the local section ${\bf{s}}_{U} \colon U^{-}(B) \to {\rm{SL}}_{3}(\mathbbm{C})$, such that ${\bf{s}}_{U}(nB) = n$, $\forall nB \in U^{-}(B)$, and considering
\begin{center}
$V(\varpi_{\alpha_{1}}) = \mathbbm{C}^{3}$ \ \  and \ \ $V(\varpi_{\alpha_{2}}) = \bigwedge^{2}(\mathbbm{C}^{3}),$
\end{center}
where $v_{\varpi_{\alpha_{1}}}^{+} = e_{1}$, and $v_{\varpi_{\alpha_{2}}}^{+} = e_{1} \wedge e_{2}$, fixed $||\cdot||$ defined by the standard ${\rm{SU}}(3)$-invariant inner product on $\mathbbm{C}^{3}$ and $\bigwedge^{2}(\mathbbm{C}^{3})$, we obtain
\begin{itemize}
\item ${\bf{h}}_{1}|_{U^{-}(B)} = \displaystyle \frac{v\overline{v}}{\bigg ( 1 + \displaystyle \sum_{i = 1}^{2}|z_{i}|^{2} \bigg )^{k} \bigg (1 + |z_{3}|^{2} + \bigg | \det \begin{pmatrix}
z_{1} & 1  \\                  
z_{2}  & z_{3} 
 \end{pmatrix} \bigg |^{2} \bigg )^{k}}$;
 \item ${\bf{h}}_{2}|_{U^{-}(B)} = \frac{ \bigg ( 1 + \displaystyle \sum_{i = 1}^{2}|z_{i}|^{2} \bigg )^{\ell}  w \overline{w}} {\displaystyle \bigg (1 + |z_{3}|^{2} + \bigg | \det \begin{pmatrix}
z_{1} & 1  \\                  
z_{2}  & z_{3} 
 \end{pmatrix} \bigg |^{2} \bigg )^{\ell}}.$
\end{itemize}
From above we obtain an explicit local description for the Hermitian structure ${\bf{h}}$ described in Eq. (\ref{locHermitian}).
\end{remark}
Therefore, by construction, it follows that 
\begin{equation}
\frac{\sqrt{-1}}{2\pi}F({\bf{h}}_{1}) = \frac{k}{4\pi} \rho_{0} = \frac{k}{2}\vartheta_{0} \ \ \ {\text{and}} \ \ \ \frac{\sqrt{-1}}{2\pi}F({\bf{h}}_{2}) = \ell \Big ({\bf{\Omega}}_{\alpha_{2}} -  {\bf{\Omega}}_{\alpha_{1}} \Big ).
\end{equation}
Considering the associated unitary frame bundle $T^{2} \hookrightarrow {\rm{U}}({\bf{E}}) \to {\mathbbm{P}}(T_{{\mathbbm{P}^{2}}})$, we can take a principal connection
\begin{equation}
\Theta = \begin{pmatrix} \sqrt{-1} \Theta_{1} & 0 \\
 0 & \sqrt{-1} \Theta_{2}\end{pmatrix} \in \Omega^{1}\big (U({\bf{E}}),{\rm{Lie}}(T^{2}) \big),
\end{equation}
such that 
\begin{equation}
{\rm{d}} \Theta_{1} = \pi^{\ast}\Big ( \sqrt{-1} F({\bf{h}}_{1})\Big ) \ \ \ {\text{and}} \ \ \ {\rm{d}} \Theta_{2} = \pi^{\ast}\Big ( \sqrt{-1} F({\bf{h}}_{2})\Big ),
\end{equation}
where $\pi \colon {\rm{U}}({\bf{E}}) \to {\mathbbm{P}}(T_{{\mathbbm{P}^{2}}})$ is the bundle projection map. It is worth observing that 
\begin{equation}
\Theta|_{U} = \pi^{\ast}\Bigg [\frac{1}{2}{\rm{d}}^{c} \log \begin{pmatrix} ||{\bf{s}}_{U}v_{\varpi_{\alpha_{1}}}^{+}||^{2k}||{\bf{s}}_{U}v_{\varpi_{\alpha_{2}}}^{+}||^{2k} & 0 \\
 0 & \displaystyle \frac{ ||{\bf{s}}_{U}v_{\varpi_{\alpha_{2}}}^{+}||^{2\ell}} {||{\bf{s}}_{U}v_{\varpi_{\alpha_{1}}}^{+}||^{2\ell}}\end{pmatrix}\Bigg ] + a_{U}^{-1}{\rm{d}}a_{U},
\end{equation}
where $a_{U} = {\text{diag}} \{a_{1},a_{2}\}$ (fiber coordinates). From above, we can equip ${\rm{U}}({\bf{E}})$ with a Hermitian structure $(\Omega,\mathbbm{J})$, which can be described as follows:
\begin{enumerate}
\item[(a)] $\mathbbm{J}|_{\ker(\Theta)}$ is given by the lift of the complex structure of ${\mathbbm{P}}(T_{{\mathbbm{P}^{2}}})$;
\item[(b)] $\mathbbm{J}(\Theta_{1}) = -\Theta_{2}$;
\item[(c)] $\Omega :=  \pi^{\ast}(\omega_{0}) + \frac{1}{2}{\rm{tr}} \big ( \Theta \wedge \mathbbm{J} \Theta\big ) = \frac{3}{8}(1-t)k^{2}\pi^{\ast}(\rho_{0}) + \Theta_{1} \wedge \Theta_{2}.$
\end{enumerate}
In this case, the Ricci form $\rho(\Omega,t)$ associated with the $t$-Gauduchon connection $\nabla^{t}$ satisfies
\begin{equation}
\rho(\Omega,t) = \pi^{\ast} \bigg ( \rho_{\omega_0} + \frac{t-1}{2} \sqrt{-1}\Lambda_{\omega_{0}}(F({\bf{h}}_{1})) \sqrt{-1}F({\bf{h}}_{1}) + \frac{t-1}{2} \sqrt{-1}\Lambda_{\omega_{0}}(F({\bf{h}}_{2})) \sqrt{-1}F({\bf{h}}_{2})\bigg ). 
\end{equation}
Since $\rho_{\omega_0} = \rho_{0} = 2\pi \vartheta_{0}$, $\sqrt{-1}F({\bf{h}}_{1}) = \pi k\vartheta_{0}$, and 
\begin{equation}
\displaystyle \omega_{0} = \frac{3}{8}(1-t)k^{2}\rho_{0}  =  \frac{3}{4}(1-t)k^{2}\pi \vartheta_{0} \Longrightarrow \Lambda_{\omega_{0}}(-) = \frac{4}{3\pi(1-t)k^{2}}\Lambda_{\vartheta_{0}}(-),
\end{equation}
it follows that 
\begin{enumerate}
\item[(F1)]$\displaystyle \sqrt{-1}\Lambda_{\omega_{0}}(F({\bf{h}}_{1})) =  \frac{4}{3(1-t)k}\Lambda_{\vartheta_{0}}(\vartheta_{0}) = \frac{4}{(1-t)k};$
\item[(F2)]$\displaystyle \sqrt{-1}\Lambda_{\omega_{0}}(F({\bf{h}}_{2})) = \frac{4\sqrt{-1}}{3\pi(1-t)k}\Lambda_{\vartheta_{0}}(F({\bf{h}}_{2})) = 0.$ 
\end{enumerate}
From above, we obtain 
\begin{equation}
\rho(\Omega,t) = \pi^{\ast} \bigg (  2\pi \vartheta_{0} + \frac{t-1}{2} \frac{4}{(1-t)k}\pi k\vartheta_{0}\bigg) = 0,
\end{equation}
i.e., $({\rm{U}}({\bf{E}}),\Omega,\mathbbm{J})$ is $t$-Gauduchon Ricci-flat ($t < 1$). In particular, we have
\begin{equation}
\rho(\Omega,1) = \pi^{\ast}(\rho_{0}) = \frac{2}{k}{\rm{d}}\Theta_{1} \Rightarrow c_{1}({\rm{U}}({\bf{E}})) = \bigg [ \frac{\rho(\Omega,1)}{2\pi}\bigg ] = 0,
\end{equation}
in other words, ${\rm{U}}({\bf{E}})$ is Calabi-Yau. Further, we obtain the following:
\begin{enumerate}
\item If $t = -1$, then $({\rm{U}}({\bf{E}}),\Omega, \mathbbm{J})$ is Strominger-Bismut Ricci-flat; 
\item If $t = 0$, then $({\rm{U}}({\bf{E}}),\Omega,\mathbbm{J})$ is Lichnerowicz Ricci-flat.
\end{enumerate}
In order to construct example of balanced Hermitian metrics, we proceed as follows. Let ${\bf{F}}_{1},{\bf{F}}_{2} \in{\rm{Pic}}^{0}_{\vartheta_{0}}({\mathbbm{P}}(T_{{\mathbbm{P}^{2}}}))$, such that 
\begin{equation}
{\bf{F}}_{1} =  \mathscr{O}_{\alpha_{1}}(-a)\otimes \mathscr{O}_{\alpha_{2}}(a) \ \ \ {\text{and}} \ \ \ {\bf{F}}_{2} =  \mathscr{O}_{\alpha_{1}}(-b)\otimes \mathscr{O}_{\alpha_{2}}(b), 
\end{equation}
where $a,b \in \mathbbm{Z}$, $a \neq 0$, $b \neq 0$, and $a \neq b$. From above, we define
\begin{equation}
{\bf{F}}:= {\bf{F}}_{1} \oplus {\bf{F}}_{2},
\end{equation}
and consider the Hermitian structure ${\bf{h}} = {\bf{h}}_{1} \oplus {\bf{h}}_{2}$, such that ${\bf{h}}_{1}$ and ${\bf{h}}_{2}$ are the Hermitian structures on ${\bf{F}}_{1}$ and ${\bf{F}}_{2}$ given, respectively, by 
$$
{\bf{h}}_{1} \myeq\frac{ \bigg ( 1 + \displaystyle \sum_{i = 1}^{2}|z_{i}|^{2} \bigg )^{a}v \overline{v}} {\displaystyle \bigg (1 + |z_{3}|^{2} + \bigg | \det \begin{pmatrix}
z_{1} & 1  \\                  
z_{2}  & z_{3} 
 \end{pmatrix} \bigg |^{2} \bigg )^{a}}  \ \ \ {\text{and}} \ \ \ {\bf{h}}_{2} \myeq\frac{ \bigg ( 1 + \displaystyle \sum_{i = 1}^{2}|z_{i}|^{2} \bigg )^{b}w \overline{w}} {\displaystyle \bigg (1 + |z_{3}|^{2} + \bigg | \det \begin{pmatrix}
z_{1} & 1  \\                  
z_{2}  & z_{3} 
 \end{pmatrix} \bigg |^{2} \bigg )^{b}},$$
for the above local description, see for instance Remark \ref{HYMinstantonlinebundle}. Considering the associated unitary frame bundle $T^{2} \hookrightarrow {\rm{U}}({\bf{F}}) \to {\mathbbm{P}}(T_{{\mathbbm{P}^{2}}})$, we can take a principal connection
\begin{equation}
\Theta = \begin{pmatrix} \sqrt{-1} \Theta_{1} & 0 \\
 0 & \sqrt{-1} \Theta_{2}\end{pmatrix} \in \Omega^{1}\big ({\rm{U}}({\bf{F}}),{\rm{Lie}}(T^{2}) \big),
\end{equation}
such that 
\begin{equation}
{\rm{d}} \Theta_{1} = \pi^{\ast}\Big ( \sqrt{-1} F({\bf{h}}_{1})\Big ) \ \ \ {\text{and}} \ \ \ {\rm{d}} \Theta_{2} = \pi^{\ast}\Big ( \sqrt{-1}  F({\bf{h}}_{2})\Big ),
\end{equation}
where $\pi \colon {\rm{U}}({\bf{F}}) \to {\mathbbm{P}}(T_{{\mathbbm{P}^{2}}})$ is the bundle projection map. Now we can equip ${\rm{U}}({\bf{F}})$ with the Hermitian structure $(\Omega,\mathbbm{J})$ described as follows:
\begin{enumerate}
\item[(a)] $\mathbbm{J}|_{\ker(\Theta)}$ is given by the lift of the complex structure of ${\mathbbm{P}}(T_{{\mathbbm{P}^{2}}})$;
\item[(b)] $\mathbbm{J}(\Theta_{1}) = -\Theta_{2}$;
\item[(c)] $\Omega =  \pi^{\ast}(\vartheta_{0}) + \frac{1}{2}{\rm{tr}} \big ( \Theta \wedge \mathbbm{J} \Theta\big ) = \pi^{\ast}(\vartheta_{0}) + \Theta_{1} \wedge \Theta_{2}.$
\end{enumerate}
By construction, we have
\begin{equation}
\delta \Omega = \pi^{\ast} (\delta \vartheta_{0}) +  \pi^{\ast} \big (\sqrt{-1}\Lambda_{\vartheta_{0}}( F({\bf{h}}_{1})) \big ) \Theta_{1} + \pi^{\ast} \big (\sqrt{-1}\Lambda_{\vartheta_{0}}( F({\bf{h}}_{2})) \big ) \Theta_{2}.
\end{equation}
Since $\delta \vartheta_{0} = 0$, and
\begin{equation}
\sqrt{-1}\Lambda_{\vartheta_{0}}( F({\bf{h}}_{1})) = \sqrt{-1}\Lambda_{\vartheta_{0}}( F({\bf{h}}_{2})) = 0,
\end{equation}
we conclude that $\delta \Omega = 0$, i.e., $({\rm{U}}({\bf{F}}),\Omega,\mathbbm{J})$ is a balanced Hermitian manifold.
\bibliographystyle{alpha}
\bibliography{bibliografia}

\end{document}